\documentclass[a4paper, 10pt, leqno]{article}

\usepackage[T1]{fontenc}       
\usepackage[utf8]{inputenc}    
\usepackage[english]{babel}    
\usepackage{lmodern}           
\usepackage{pgfplots}
\pgfplotsset{compat=1.18}
\usepackage{amsmath}           
\usepackage{amssymb}           
\usepackage{amsthm}            
\usepackage{empheq}            
\usepackage{mathtools}         
\usepackage{mathrsfs}          
\usepackage{eucal}             
\usepackage{braket}            
\usepackage{mathtools}         
\mathtoolsset{showonlyrefs}    

\theoremstyle{plain}           
\newtheorem{thm}{Theorem}[section]            
\newtheorem{cor}[thm]{Corollary}              
\newtheorem{lem}[thm]{Lemma}                  
\newtheorem{prop}[thm]{Proposition}           

\theoremstyle{definition}      
\newtheorem{defn}[thm]{Definition}            

\theoremstyle{remark}         
\newtheorem{rem}[thm]{Remark}                 

\numberwithin{equation}{section}  
\usepackage{a4wide}              
\usepackage{mparhack}            
\usepackage{relsize}             
\usepackage{footmisc}            

\usepackage{booktabs}            
\usepackage{multirow}            
\usepackage{caption}             
\usepackage{subfig}              
\usepackage{rotating}            
\usepackage{graphicx}            
\usepackage{epsfig}              
\usepackage[all,pdf]{xy}         

\usepackage[numbers]{natbib}     
\usepackage{varioref}            
\usepackage[
    colorlinks=true,             
    linkcolor=black,             
    citecolor=black,             
    urlcolor=blue,               
    pdfborder={0 0 0}            
]{hyperref}

\usepackage{enumerate}           

\usepackage{xcolor}
\usepackage{xcolor}
\usepackage{amssymb}
\usepackage{tikz}

\makeatletter
\newcommand{\twodigits}[1]{\ifnum#1<10 0\the#1\else\the#1\fi}

\newcommand{\now}{%
  \the\year-\twodigits{\month}-\twodigits{\day} \space%
  \count0=\time \divide\count0 by 60 \twodigits{\count0}:%
  \count0=\time \count1=\count0 \divide\count1 by 60 \multiply\count1 by 60 %
  \advance\count0 by -\count1 \twodigits{\count0}%
}

\newcommand{\checkpoint}[2][The proof continues here...]{%
  \par\vspace{2em}
  \noindent\begin{center}
    \begin{tikzpicture}
      \fill[gray!20] (0.2,-0.2) rectangle (12.2, 3.2);
      \draw[fill=yellow!10, draw=orange!50!black, thick] (0,0) rectangle (12,3);
      
      \fill[red!70!black] (6, 3.1) circle (0.15);
      \draw[red!70!black, ultra thick] (6, 3.1) -- (6, 2.8);
      
      \node[anchor=west] at (0.3, 2.6) {\large\textbf{\textsf{\color{orange!60!black}MISSION LOG}}};
      \node[anchor=east] at (11.7, 2.6) {\tiny\texttt{LOC: LINE \the\inputlineno}};
      
      \node[anchor=north west, text width=11cm] at (0.5, 2.3) {
        \small \textbf{Summary:} #2
      };
      
      \node[anchor=north west, text width=11cm] at (0.5, 1.3) {
        \small \textbf{Next Relay:} \textcolor{blue!60!black}{\itshape #1}
      };
      
      \draw[gray!30] (0.5, 0.5) -- (11.5, 0.5);
      \node[anchor=east] at (11.7, 0.25) {\tiny\textsc{Authored at: \now}};
    \end{tikzpicture}
  \end{center}
  \vspace{2em}
}
\makeatother

\newcommand{\inpro}[2]{\left\langle #1,#2\right\rangle} 
\newcommand{\spfl}[1]{\operatorname{sf}\left(#1\right)} 
\newcommand{\email}[1]{\href{mailto:#1}{\textsf{#1}}}   
\newcommand{\Id}{\operatorname{Id}} 

\renewcommand{\geq}{\geqslant}
\renewcommand{\hat}{\widehat}
\renewcommand{\tilde}{\widetilde}

\newcommand{\Z}{\mathbb{Z}}        
\newcommand{\R}{\mathbb{R}}        
\newcommand{\C}{\mathbb{C}}        

\let\d\relax                       
\newcommand{\d}{\mathrm{d}}

\DeclareMathOperator{\sgn}{sgn}      
\DeclareMathOperator{\sign}{sign}    
\DeclareMathOperator{\dom}{dom}      

\newcommand{\Lagr}{\Lambda}               
\newcommand{\Lag}[1]{\Lagr({#1})}         

\newcommand{\iMor}{\mathrm{m^-}}                
\newcommand{\coiMor}{\mathrm{m^+}}              
\newcommand{\iCLM}{\mu^{\scriptscriptstyle{\mathrm{CLM}}}} 

\newcommand{\Sym}[1]{\mathrm{Sym}(#1,\R)}       
\newcommand{\Real}{\mathrm{Re}}                 

\usepackage{bm} 

\newcommand{\noo}[1]{\overset {\mbox{%
\lower1pt\hbox{${\scriptscriptstyle o}$}}}n^{\mbox{%
\lower2pt\hbox{$\scriptscriptstyle #1$}}}} 
\newcommand{\cic}[1]{ C_{#1}}               
\newcommand{\defeq}{\coloneq}

\title{Instability of the Standing Pulse in Skew-Gradient Systems and Its Application to FitzHugh-Nagumo Type Systems}

\author{Jing Li,\quad Qin Xing\thanks{The author is supported  by  the National Natural Science Foundation of China (No. 12201278) and the Natural Science Foundation of Shandong Province (No. ZR2022QA076)},\quad Ran Yang \thanks{The author is supported by  the National Natural Science Foundation of China (No. 12001098, No. 42264007) and the  Natural Science Foundation of Jiangxi Province(No.20232BAB211005)}}

\date{\today}

\date{\today}
\begin{document}
 \maketitle

\begin{abstract}
  Classical results from Sturm-Liouville theory establish that the Morse index of a one-dimensional Sturm-Liouville operator defined on $\mathbb{R}$ is equal to the number of its associated conjugate points. Recent advancements by Beck et al.~\cite{BCJLM18} have extended these results to higher-dimensional Sturm-Liouville operators on $\mathbb{R}$, utilizing the Maslov index to characterize the spectral stability of nonlinear waves in multi-component systems. 
In this paper, we extend this framework further to non-self-adjoint settings by investigating skew-gradient reaction-diffusion systems. By utilizing the Maslov index and spectral flow, we derive an instability criterion for standing pulses. This approach bridges the gap between variational methods and the stability index in systems where the standard self-adjoint structure is absent. 
As a primary application, we apply our results to FitzHugh-Nagumo type systems, where the reaction terms for both the activator and inhibitor exhibit intrinsic nonlinearities. This provides a robust topological method to account for the influence of nonlinear inhibition on pulse stability in the non-self-adjoint regime.
	
\end{abstract}


\section{Introduction and description of the problem}

The stability of localized structures in reaction-diffusion systems has been a subject of intense investigation for decades, driven by its fundamental role in pattern formation \cite{evans1972nerve, yanagida1985stability, flores1991stability, van2008pulse, CH14, BCJLM18}. This line of inquiry originated from the seminal work of Turing \cite{turing1990chemical}, who demonstrated that spatially periodic structures---now known as {Turing patterns}---can emerge from a uniform background state through a symmetry-breaking instability driven by diffusion.

A pervasive observation in physical and biological systems exhibiting such patterns is that the diffusion coefficients of the interacting species are unequal, often differing by several orders of magnitude. This empirical evidence has fostered the long-standing conjecture that  {differential diffusion} is a necessary condition for the stability, and thus the physical realization, of Turing patterns. Mathematically, the stability of periodic patterns is often intimately linked to the robustness of their constituent localized  {pulse solutions}. Consequently, proving the inherent instability of pulses in the absence of diffusion disparity  offers a rigorous pathway toward validating the Turing pattern conjecture for periodic structures \cite{BCJLM18}.

Motivated by these considerations, this paper investigates the spatio-temporal dynamics and stability properties of a class of reaction-diffusion equations characterized by a  {skew-gradient structure}:
\begin{align}\label{eq:r.d.eq}
    Mw_t = D w_{xx} + Q \nabla V(w),
\end{align}
where $w(x,t) = (w_1, \dots, w_n)^T \in \mathbb{R}^n$ denotes the state vector. The diagonal matrices $M = \text{diag}(m_i)$ and $D = \text{diag}(d_i)$ represent the  {temporal response rates} and  {spatial diffusion coefficients}, respectively. The reaction kinetics are governed by the gradient of a scalar potential $V: \mathbb{R}^n \rightarrow \mathbb{R}$, modulated by:
\[
Q = \begin{pmatrix}
    \Id_j & 0 \\
    0 & -\Id_{n-j}
\end{pmatrix}.
\]
As introduced by Yanagida \cite{yanagida2002mini}, this configuration defines a  {skew-gradient system} that effectively captures the essential features of  {activator-inhibitor} dynamics. Here, the first $j$ components act as activators, evolving toward increasing values of $V$, while the remaining $n-j$ components function as inhibitors that counteract this growth.

The presence of the matrix $Q$ is pivotal for the existence of  {standing pulses}. In a conventional gradient system (where $Q = \Id$), the dynamics are driven by energy minimization, typically resulting in the coarsening of domains or a collapse to global equilibria. In contrast, the skew-gradient framework introduces an inhibitory feedback loop that can stabilize localized excitations against the homogenizing effect of diffusion. This non-self-adjoint structure not only allows for the existence of pulse solutions but also provides a rich mathematical setting to analyze how the interplay between the potential $V$ and uniform diffusion $D$ leads to the destabilization of these structures.

As in \eqref{eq:r.d.eq}, we assume that \( w \equiv 0 \) is always an equilibrium solution of \eqref{eq:r.d.eq}. A standing pulse solution of \eqref{eq:r.d.eq} is a non-constant wave function \( w \) satisfying
\begin{align}\label{eq:state equation}
	\begin{cases}
		D  w'' + Q \nabla V(w) = 0, \\
		\lim\limits_{|x| \to \infty}  w(x) = \lim\limits_{|x| \to \infty} \ w'(x) = 0,
	\end{cases}
\end{align}
where throughout this paper, $'$ denotes differentiation with respect to \( x \).  and define \( \C^{-} = \{ z \in \C \mid \Real z < 0 \} \), where \( \Real z \) denotes the real part of \( z \).

Let \( w_0 \) be a standing pulse solution of \eqref{eq:state equation}. The stability of \( w_0 \) is determined by the equation
\begin{align}\label{eq:Weighted eigenvalue pro.}
    D \phi''(x) + Q B(x) \phi(x) = \lambda M \phi(x)
\end{align}
or its equivalent eigenvalue problem
\begin{align}\label{eq:eigenvalue pro.}
	\mathcal{L} \phi(x) = \lambda \phi(x),
\end{align}
where the operator \( \mathcal{L} \) is given by
\[
\mathcal{L} = M^{-\frac{1}{2}} \left( D \frac{\d^2}{\d x^2} + Q B(x) \right) M^{-\frac{1}{2}},
\]
with \( B(x) = \nabla^2 V\left( w_0 \right) \). Therefore, the limit $\lim\limits_{|x| \to \infty} B(x) = B(\infty)$ is well-defined. Additionally, there exists a constant \(C_1 > 0\) such that
\begin{align}\label{eq:constant C_1}
	|\langle Q B(x)  v, v \rangle| \leq C_1 |v|^2 \quad \text{for all} \ (x, v) \in \mathbb{R} \times \mathbb{R}^n.
\end{align}

Since \eqref{eq:state equation} is an autonomous system, the translation invariance of \( w_0(x) \) with respect to \( x \) implies that zero is always an eigenvalue of \eqref{eq:eigenvalue pro.}, with \(  w'_0(x) \) as the corresponding eigenfunction. In such a system, a standing pulse \( w_0(x) \) is said to be non-degenerate if zero is a simple eigenvalue of \eqref{eq:eigenvalue pro.}. 

\begin{defn}\cite{CH14}\label{def:stability}
	A non-degenerate standing pulse \( w_0 \) is \emph{spectrally stable} if all non-zero eigenvalues of \eqref{eq:eigenvalue pro.} lie in \( \mathbb{C}^- \).
\end{defn}

This paper investigates the stability of standing pulses of \eqref{eq:r.d.eq}. Much of the analysis relies on understanding the detailed spectral properties of the linearized operator $\mathcal{L}$ obtained by linearizing the reaction-diffusion equation along a standing pulse~\cite{jones1990topological,BCJLM18,cornwell2018existence,X21}. Over the past few decades, the Evans function has been the primary tool for analyzing the spectral properties of linearized operators in a wide range of partial differential equations, including reaction-diffusion systems, as well as other classes such as nonlinear Schrödinger equations and KdV-type equations~\cite{jones1985instability,jones1990topological,pego1992eigenvalues,sandstede2002stability,zumbrun2005stability,barker2018evans}.

 Yanagida\cite{yanagida2002mini}  introduced the skew-gradient structure to characterize systems that resemble gradient systems but possess nonlinearities with mixed signs, effectively modeling the complex dynamics of activator-inhibitor pairs. By defining the geometrical orientation of a pulse and employing the Evans function, Yanagida established a fundamental instability criterion: a stationary pulse becomes unstable when the temporal response rates (time constants) of the inhibitory components are sufficiently large relative to the activators. This seminal work provided the theoretical basis for understanding how the interplay between inhibitory feedback and disparate time scales drives the destabilization of localized patterns. Building upon this framework, Chen and Hu \cite{CH14}investigated standing pulse solutions in FitzHugh-Nagumo equations and more general skew-gradient systems. They observed that standing pulses in such configurations correspond to homoclinic orbits of a second-order Hamiltonian system, allowing for the application of Maslov index  to provide a rigorous stability analysis. Their work bridged the gap between the variational structure of the reaction terms and the spectral properties of the linearized operator, establishing a robust framework for determining the stability of localized patterns.
 
 Since the linearized problem \eqref{eq:Weighted eigenvalue pro.} inherits a Hamiltonian structure, the Maslov index has emerged as a powerful alternative tool for analyzing spectral properties \cite{CH07, CH14, cornwell2018existence}. The application of the Maslov index to the stability of solitary waves was pioneered by Jones \cite{jones1985instability} and Bose and Jones \cite{bose1995stability}, offering a higher-dimensional generalization of Sturmian theory. In this paper, we extend these methodologies by utilizing the index theory developed in \cite{HP17} and \cite{HPWX20}. We employ the Maslov index--which, in this specific framework, admits only positive crossing forms-- to establish a lower bound on the number of unstable eigenvalues for the operator $\mathcal{L}$ defined in \eqref{eq:eigenvalue pro.}. This bound serves as a fundamental instability criterion for the system \eqref{eq:r.d.eq}. Furthermore, we utilize spectral flow to analyze the evolution of the zero eigenvalue in response to variations in the temporal response rates $M$. By tracking these spectral crossings, we establish an additional instability result that complements the Maslov index approach, providing a more comprehensive understanding of the spatio-temporal dynamics in skew-gradient systems.

Let \( z(x) = (QD\psi', \psi)^\top \). Then, the weighted eigenvalue problem \eqref{eq:Weighted eigenvalue pro.} can be transformed into the following linear Hamiltonian system:
\begin{equation}\label{eq:linear hamiltonian system}
    \begin{cases}
        z'(x) = J A_\lambda(x) z(x), \quad x \in \mathbb{R}, \\
        \lim\limits_{|x| \to \infty} z(x) = 0,
    \end{cases}
\end{equation}
where the coefficient matrix is given by
\[
A_\lambda(x) = \begin{pmatrix}
    (QD)^{-1} & 0 \\
    0 & B(x) - \lambda Q M
\end{pmatrix}.
\]
We note that the asymptotic matrix \( A_\lambda(\infty) \defeq \lim_{|x| \to \infty} A_\lambda(x) \) is well-defined. Let \( \mathcal{F}_\lambda = -J \frac{\d}{\d x} - A_\lambda(x) \) denote the associated Hamiltonian differential operator.
In this paper, we impose the following structural assumptions:
\begin{enumerate}
    \item[ {(H1)}] For all \( v \in \mathbb{R}^n \setminus \{0\} \), it holds that \( \langle QB(\infty)v, v \rangle < 0 \).
    \item[ {(H2)}] For all \( x \in \mathbb{R} \) and \( v \in V^{-}(Q) \setminus \{0\} \), it holds that \( \langle B(x)v, v \rangle > 0 \), where $V^-(Q)$ denote the  negative spectral subspaces spanned by the eigenvectors corresponding to eigenvalues with  negative real parts.
\end{enumerate}

\noindent In conclusion, conditions  {(H1)} and  {(H2)} are not merely technical assumptions but are  {naturally satisfied} by a broad class of physically relevant activator-inhibitor systems.  {(H1)} ensures the  {hyperbolicity} of the associated linear Hamiltonian system at infinity, reflecting the fundamental requirement that the background state must be stable for localized patterns to persist. Simultaneously,  {(H2)} guarantees the global consistency of the inhibitory feedback mechanism, which manifests as the monotonicity of the Lagrangian subspace evolution. Together, these conditions allow the Maslov index to rigorously bridge the gap between the topological properties of the phase space and the spectral instability of the standing pulse.

\begin{rem}\label{rem:C_2 and C_3}
Under assumption  {(H1)}, since the quadratic form \( v \mapsto \langle QB(\infty)v, v \rangle \) is continuous on the unit sphere \( S^{n-1} \defeq \{ v \in \mathbb{R}^n : |v|=1 \} \), it attains a strictly negative maximum. Consequently, there exists a constant \( C_2 > 0 \) such that
\[
\langle QB(\infty)v, v \rangle \leq -C_2 |v|^2, \qquad \forall\, v \in \mathbb{R}^n.
\]
Similarly, if  {(H2)} holds, we consider the function 
\[
F(x, v) = \langle B(\tan x)v, v \rangle, \quad (x, v) \in \left[-\frac{\pi}{2}, \frac{\pi}{2}\right] \times \{v \in V^-(Q) : |v|=1 \}.
\]
Since \( F(x, v) \) is continuous on a compact set, there exists a constant \( C_3 > 0 \) such that \( F(x, v) \geq C_3 \), implying
\[
\langle B(x)v, v \rangle \geq C_3|v|^2
\]
for all \( x \in \mathbb{R} \) and \( v \in V^-(Q) \).
\end{rem}

Let \( \Phi_{\lambda}(x, \tau) \) be the fundamental matrix solution of \eqref{eq:linear hamiltonian system} satisfying \( \Phi_{\lambda}(\tau, \tau) = I_{2n} \). The stable and unstable subspaces at \( \tau \in \mathbb{R} \) are defined as:
\[
E_\lambda^s(\tau) = \left\{ \xi \in \mathbb{R}^{2n} : \lim_{x \to +\infty} \Phi_{\lambda}(x, \tau) \xi = 0 \right\},
\]
\[
E_\lambda^u(\tau) = \left\{ \xi \in \mathbb{R}^{2n} : \lim_{x \to -\infty} \Phi_{\lambda}(x, \tau) \xi = 0 \right\}.
\] 
For brevity, we write \( E^s(\tau) \) and \( E^u(\tau) \) for \( E_0^s(\tau) \) and \( E_0^u(\tau) \), respectively. As shown in \cite[Lemma 3.1]{HP17}, these subspaces \( E_\lambda^s(\tau) \) and \( E_\lambda^u(\tau) \) are Lagrangian for all \( (\tau, \lambda) \in \mathbb{R} \times [0, +\infty) \).

Let \(\sigma_p(\mathcal{L})\) denote the set of isolated eigenvalues of \(\mathcal{L}\) with finite multiplicity, and let \(\sigma_{\text{ess}}(\mathcal{L}) = \sigma(\mathcal{L}) \setminus \sigma_p(\mathcal{L})\) represent the essential spectrum of \(\mathcal{L}\). According to \cite[Lemma 3.1.10]{kapitula2013spectral}, the essential spectrum is characterized by
\[
\sigma_{\text{ess}}(\mathcal{L}) = \left\{\lambda \in {\C} \mid A_\lambda(\infty) \text{ has an eigenvalue } \mu \in i \mathbb{R} \right\}.
\]

From Lemma \ref{lem:matrix hyperbolic}, under (H1), we have that $\sigma_{\text{ess}}(\mathcal{L})\subset C^-\defeq\{z\in \Z|\Real z<0\}$. So, to analysis the stable of the standing pulse, we only consider $\sigma(\mathcal{L})$.

\begin{defn}\label{def:stability index}
    Let \( w_0 \) be a standing pulse of \eqref{eq:r.d.eq}. Define the stable index of \( w_0 \) as
    \[
    i(w_0) := \sum_{\tau \in \mathbb{R}}\dim \left( \Lambda_R \cap E^u(\tau) \right),
    \]
    where \( \Lambda_R = \left\{ \left. \begin{pmatrix}
        p\\q
    \end{pmatrix} \right| p \in V^+(Q),\; q \in V^-(Q) \right\} \).
\end{defn}

We now establish the following instability criterion for the standing pulses of the reaction-diffusion system \eqref{eq:r.d.eq}.

\begin{thm}\label{thm:mainly result}
    Under conditions (H1) and (H2), let \( w_0 \) be a standing pulse of \eqref{eq:r.d.eq}. The pulse \( w_0 \) is spectrally unstable if either of the following conditions is satisfied:
    \begin{itemize}
        \item[(1)] The stability index $i(w_0)$ is strictly positive, i.e., \( i(w_0) > 0 \).
        \item[(2)] $\int_{-\infty}^{\infty}\inpro{QMw'_0(x)}{w'_0(x)}\d x<0$.
    \end{itemize}
\end{thm}

As an application, the related result is applied to the following FitzHugh-Nagumo type system:
\begin{subequations}\label{eq:f.n.eq.}
    \begin{empheq}[left=\empheqlbrace]{align}
      u_t &= d u_{xx} + f(u) - v, \label{eq:f.n.eq.1} \\
      \tau v_t &= v_{xx} - \gamma v - v^3 + u, \label{eq:f.n.eq.2}
    \end{empheq}
\end{subequations}
where \( f(u) = u(1 - u)(u - \beta) \), and \( d, \tau, \gamma \), and \( \beta \) are positive constants. This system is of activator-inhibitor type, with nonlinear structures inherent in the reaction terms of both the activator and inhibitor. 

Observe that the system \eqref{eq:f.n.eq.1}--\eqref{eq:f.n.eq.2} possesses a skew-gradient structure with potential:
\[
V(u,v) = \frac{1}{2} \gamma v^2 + \frac{1}{4} v^4 - u v - \frac{1}{4} u^4 + \frac{1}{3}(1 + \beta) u^3  -\frac{1}{2} \beta u^2.
\]

The calculus of variations was employed in \cite{choi2021existence} to establish the existence of standing pulses for \eqref{eq:f.n.eq.1} and \eqref{eq:f.n.eq.2}. 
As an application of case (1) in Theorem \ref{thm:mainly result}, we establish the following instability result:

\begin{thm}\label{thm:unstable by solution}
     Let $w_0=(u,v)^\top$ be a standing pulse of \eqref{eq:f.n.eq.1} and \eqref{eq:f.n.eq.2}. If there exists a point $x_0 \in \mathbb{R}$ such that $u'(x_0)=v''(x_0)=0$, then the pulse $w_0$ is unstable.
\end{thm}

\begin{rem}
    The significance of Theorem \ref{thm:unstable by solution} lies in its ability to predict instability directly from the geometric profile of the standing pulse without requiring explicit spectral computation. The coincidence of the activator's peak ($u'=0$) and the inhibitor's inflection point ($v''=0$) suggests a localized failure of the inhibitory feedback mechanism. In such a geometric configuration, the restorative force provided by the inhibitor is insufficient to counterbalance the autocatalytic growth of $u$, leading to the migration or collapse of the pulse profile. This criterion differs from results for gradient-type systems; for instance, in \cite{BCJLM18}, it is proven that if there exists $x_0$ such that $u'(x_0)=v'(x_0)=0$, then the standing pulse is unstable. Our result demonstrates that for non-gradient systems, the higher-order derivative of the inhibitor plays a critical role in the stability index.
\end{rem}

\begin{figure}[h]
\centering
\begin{tikzpicture}
\begin{axis}[
    width=12cm, height=7cm,
    axis lines=middle,
    xmin=-5, xmax=10,
    ymin=-1.2, ymax=1.5,
    xlabel={$x$},
    ylabel={Amplitude},
    domain=-4:6,
    samples=300, 
    smooth,
    xtick={1}, 
    xticklabels={$x_0$},
    ytick={0},
    legend style={at={(0.95,0.9)}, anchor=north east, draw=none, fill=none, font=\small},
    axis line style={thick},
]

\addplot [red, ultra thick] {1.2 * exp(-(x-1)^2 / (1.5 + 0.8*tanh(x-1)))};
\addlegendentry{Asymmetric $u(x)$}

\addplot [blue, ultra thick] {1.8 * (x-1) * exp(-(x-1)^2 / 2)};
\addlegendentry{Asymmetric $v(x)$}

\draw[dashed, gray!60] (axis cs:1,-0.2) -- (axis cs:1,1.3);

\coordinate (Upoint) at (axis cs:1, 1.2);
\draw[red!80!black, <-, thick] (Upoint) -- (axis cs:-0.5, 1.4) 
    node[left, font=\footnotesize] {$u'(x_0)=0$ (Peak)};

\node[circle, fill=blue, inner sep=1.2pt] (Vpoint) at (axis cs:1, 0) {};
\draw[blue!80!black, <-, thick] (Vpoint) -- (axis cs:2.5, -0.15) 
    node[right, font=\footnotesize] {$v''(x_0)=0$ (Inflection)};

\node[font=\tiny, color=gray] at (axis cs:-3.5, 0.1) {$0\gets u,v $};
\node[font=\tiny, color=gray] at (axis cs:5.5, 0.1) {$ u,v \to 0$};

\end{axis}
\end{tikzpicture}
\caption{Profile of a standing pulse illustrating the geometric condition $u'(x_0)=v''(x_0)=0$ for Theorem \ref{thm:unstable by solution}.}
\end{figure}

As the application of case (2) of Theorem \ref{thm:mainly result}, we obtain the following:
\begin{thm}\label{thm:unstablity condition for example}
    Let $w_0=(u,v)^\top$ be a standing pulse of \eqref{eq:f.n.eq.1} and \eqref{eq:f.n.eq.2}. Let $\tau_0=\frac{\int_{-\infty}^{\infty}|u'(x)|^2\d x}{\int_{-\infty}^{\infty}|v'(x)|^2\d x}$, then $(u,v)^\top$ is unstable if $\tau>\tau_0$.
\end{thm}
\begin{rem}
    Theorem \ref{thm:unstablity condition for example} establishes a sufficient condition for the onset of instability. Specifically, the inequality $\tau > \tau_0$ serves as a rigorous bound beyond which the standing pulse is no longer sustainable. Geometrically, this threshold signifies that the temporal relaxation of the inhibitor is too slow to effectively counteract the spatial autocatalytic growth of the activator. This leads to the emergence of instabilities such as "breathing" oscillations.
\end{rem}

By the spectral analysis introduced in \cite{CH14}, we state the following stability result:
\begin{thm}\label{thm:stablity condition for example}
    Suppose that $w_0=(u,v)^\top$ is a non-degenerate standing pulse of \eqref{eq:f.n.eq.1} and \eqref{eq:f.n.eq.2}. Then $w_0$ is stable if $i(w_0)=0$ and $\tau<\gamma^2$. 
\end{thm}

The remainder of this paper is organized as follows. In Section 2, we first establish a spectral flow formula in Proposition \ref{pro:spectral flow}, which serves as the foundation for deriving the instability criterion for \eqref{eq:r.d.eq} presented in Theorem \ref{thm:mainly result}. Subsequently, we apply the results of Theorem \ref{thm:mainly result} to conduct a detailed stability and instability analysis for the FitzHugh-Nagumo type system \eqref{eq:f.n.eq.1}--\eqref{eq:f.n.eq.2}. For the convenience of the reader, Section 3 provides a brief overview of the Maslov index, spectral flow, and their relevant  properties used throughout our analysis.

\section{Spectral flow formula}\label{section 2}
In this section, we establish the following spectral flow formula, which plays a crucial role in analyzing the stability of the standing pulse in \eqref{eq:r.d.eq}.

\begin{prop}\label{pro:spectral flow}
    Under conditions (H1) and (H2), there exists a constant $T_\infty > 0$ such that 
    \begin{align}
        \iCLM(\Lambda_R, E^u(\tau), \tau \in (-\infty, T]) = \spfl{\mathcal{F}_\lambda, \lambda \in [0, \hat\lambda]}
    \end{align}
    for all $T \geq T_\infty$, where $\hat\lambda$ is defined as in Corollary \ref{cor:nondegenerate operator}. Consequently, the stability index satisfies:
    \[
    i(w_0) = \spfl{\mathcal{F}_\lambda, \lambda \in [0, \hat\lambda]}.
    \]
\end{prop}

\begin{rem}
    A direct computation shows that $z(x) = (\psi(x), \phi(x))^\top$ is a solution to the Hamiltonian system \eqref{eq:h.s.with two parameters} if and only if $\psi(x) = QD\phi'(x)$ and $(\phi(x), \lambda)$ satisfies the weighted eigenvalue problem \eqref{eq:Weighted eigenvalue pro.}. This implies that in this context, $\lambda$ is an eigenvalue of the linearized operator $\mathcal{L}$. 
    By Corollary \ref{cor:nondegenerate operator}, for all $\lambda > \hat\lambda$, the operator $\mathcal{F}_\lambda$ is non-degenerate (i.e., $\ker \mathcal{F}_\lambda = \{0\}$). Consequently, the non-negative real eigenvalues of $\mathcal{L}$ are contained within the interval $[0, \hat{\lambda}]$. Therefore, the spectral flow $\spfl{\mathcal{F}_\lambda, \lambda \in [0, \hat\lambda]}$ captures the complete spectral information necessary to analyze the stability of the standing pulse.
\end{rem}
The proof of Proposition \ref{pro:spectral flow} is established through a sequence of technical lemmas, culminating in the final derivation following the proof of Lemma \ref{lem:sf<=0}. 

In \cite{CH14}, the authors proved that under conditions (H1) and (H2), there exists $T_\infty$ such that the following spectral flow formula holds for all $T \geq T_\infty$:
\begin{align}\label{eq:sl formula for CH}
    \iCLM(E^s(T_\infty), E^u(\tau), \tau \in (-\infty, T]) = \spfl{\mathcal{F}_\lambda, \lambda \in [0, \hat\lambda]}.
\end{align}
Our formula improves upon \eqref{eq:sl formula for CH} in that it does not require considering the direction of intersection when computing the Maslov index. Analogous to the role of the Sturmian oscillation theorem in analyzing standing pulses of scalar reaction-diffusion equations, our formula provides a geometric characterization that links the stability to the pulse's shape.

A sufficient approach to proving Proposition \ref{pro:spectral flow} would be to utilize the Hörmander index to estimate the difference between $\iCLM(\Lambda_R, E^u(\tau), \tau \in (-\infty, T_\infty])$ and $\iCLM(E^s(T_\infty), E^u(\tau), \tau \in (-\infty, T_\infty])$. However, due to the translational invariance of the standing pulse, so $E^s(0)\cap E^u(0)\neq \{0\}$, then  the explicit expression for the limit $\lim_{\tau \to \infty} E^u(\tau)$ is generally unavailable, which makes estimating this difference challenging. Nevertheless, according to Lemma \ref{lem:space convergence}, in cases where \eqref{eq:Weighted eigenvalue pro.} does not possess a zero eigenvalue (i.e., $E^u(0) \cap E^s(0) = \{0\}$), the explicit expression of $\lim_{\tau \to \infty} E^u(\tau)$ can be obtained, thereby allowing for a straightforward estimation of the difference.
To overcome the difficulties associated with the non-self-adjoint structure in the general case, we introduce the following two-parameter family of Hamiltonian systems:
\begin{align}\label{eq:h.s.with two parameters}
   \begin{cases}
     z'(x) = J A_{\lambda, \epsilon}(x) z(x), \quad x \in \mathbb{R} \\
     \lim\limits_{|x| \to \infty} z(x) = 0,
   \end{cases}
\end{align}
where the coefficient matrix is defined as
\[
A_{\lambda, \epsilon}(x) = \begin{pmatrix}
    (QD)^{-1} & 0 \\
    0 & B_\epsilon(x) - \lambda Q M
\end{pmatrix}.
\]
Here, $B_\epsilon(x) \defeq B(x) - \epsilon I$, with the parameter $\epsilon$ varying in the interval $[0, C_3]$ and $C_3 = \frac{1}{2}\min\{C_1, C_2\}$.

Let $\mathcal{F}_{\lambda, \epsilon} \defeq -J \frac{\d}{\d x} - A_{\lambda, \epsilon}(x)$ denote the Hamiltonian differential operator associated with the system \eqref{eq:h.s.with two parameters}. To maintain conciseness, we adopt the following notational conventions for cases where one parameter is set to zero:
\begin{itemize}
    \item When $\epsilon = 0$, we denote $B_0$ and $A_{\lambda, 0}$ by $B$ and $A_\lambda$, respectively, and write $\mathcal{F}_\lambda$ for $\mathcal{F}_{\lambda, 0}$.
    \item When $\lambda = 0$, we denote $\mathcal{F}_{0, \epsilon}$ by $\mathcal{F}_\epsilon$.
\end{itemize}

In the case where $\epsilon = 0$, the two-parameter Hamiltonian system \eqref{eq:h.s.with two parameters} reduces to the linear Hamiltonian system \eqref{eq:linear hamiltonian system}. This system is equivalent to the weighted eigenvalue problem \eqref{eq:Weighted eigenvalue pro.}, providing a direct link between the geometric intersection properties of the Lagrangian subspaces and the spectral stability of the standing pulse.

From Remark \ref{rem:C_2 and C_3}, it is evident that 
\begin{align}\label{eq:B epsilon bound}
    \begin{split}
        & \inpro{QB_\epsilon(\infty)v}{v} < 0, \quad \text{for all } v \in \mathbb{R}^n \setminus \{0\}, \\
        & \inpro{B_\epsilon(x)v}{v} > 0, \quad \text{for all } v \in V^-(Q) \setminus \{0\}.
    \end{split}
\end{align}
From now on, we restrict our attention to the case where $\epsilon \in [0, C_3]$ and $\lambda \geq 0$.

The roles of the two parameters involved in our analysis are distinct yet complementary. The parameter $\epsilon$ introduces a perturbation to the system \eqref{eq:linear h.s.eq.} to break the degeneracy arising from the translational invariance of the standing pulse. Specifically, this perturbation restores hyperbolicity at infinity, thereby ensuring the existence of a well-defined limit for the unstable subspace $E^u_\epsilon(\tau)$ as $\tau \to +\infty$. This asymptotic convergence is a critical prerequisite for computing the Hörmander index. Conversely, the parameter $\lambda$ serves as the spectral parameter, providing the necessary information to resolve the weighted eigenvalue problem \eqref{eq:Weighted eigenvalue pro.}.

The proof of Proposition \ref{pro:spectral flow} is organized into the following four-step strategy:

\begin{enumerate}
    \item \textbf{Stability of Spectral Flow:} We demonstrate that for sufficiently small $\epsilon > 0$, the spectral flow $\spfl{\mathcal{F}_{\lambda,\epsilon}, \lambda \in [0, \hat{\lambda}]}$ remains invariant and is equivalent to the spectral flow $\spfl{\mathcal{F}_{\lambda}, \lambda \in [0, \hat{\lambda}]}$ of the unperturbed operator (cf. Lemma \ref{lem:sf not change}).
    
    \item \textbf{Asymptotic Behavior and Hörmander Index:} For $\epsilon > 0$ sufficiently small, the transversality condition $E_\epsilon^s(0) \cap E_\epsilon^u(0) = \{0\}$ is satisfied. Utilizing the fact that $\lim_{\tau \to \infty} E^u_\epsilon(\tau) = V^+(JA_\epsilon(\infty))$, we apply the properties of the Hörmander index to establish the existence of $\hat{T}_\epsilon > 0$ such that:
    \[
    \iCLM(\Lambda_R, E^u_\epsilon(\tau), \tau \in (-\infty, T]) = \spfl{\mathcal{F}_{\lambda, \epsilon}, \lambda \in [0, \hat{\lambda}]}
    \]
    for all $T \geq \hat{T}_\epsilon$ (cf. Lemma \ref{lem:postive maslov = sf}).

    \item \textbf{Homotopy Invariance and Lower Bound:} We construct a homotopic Lagrangian path $(\Lambda_R, E^u_\epsilon(\tau))$ for $(\tau, \epsilon) \in (-\infty, T] \times [0, \hat{\epsilon}]$. By the homotopy invariance of the Maslov index and observing that $\iCLM(\Lambda_R, E^u_\epsilon(T), \epsilon \in [0, \hat{\epsilon}]) \leq 0$ (see Figure \ref{fig:homotopy 1}), we prove that there exists $T_0 > 0$ such that:
    \[
    \iCLM(\Lambda_R, E^u(\tau), \tau \in (-\infty, T]) \geq \spfl{\mathcal{F}_{\lambda}, \lambda \in [0, \hat{\lambda}]}
    \]
    for all $T \geq T_0$ (cf. Lemma \ref{lem:maslov>=sf}).

    \item \textbf{Monotonicity and Upper Bound:} We consider the homotopic Lagrangian paths $(\Lambda_R, E^u_\epsilon(\tau))$ and $(E^s_\epsilon(T), E^u_\epsilon(\tau))$ for $(\tau, \epsilon) \in (-\infty, T] \times [0, \tilde{\epsilon}]$. A key observation is that the difference 
    \[
     \Delta(T) \defeq \iCLM(\Lambda_R, E^u_{\tilde{\epsilon}}(\tau), \tau \in (-\infty, T]) - \iCLM(E^s_{\tilde{\epsilon}}(T), E^u_{\tilde{\epsilon}}(\tau), \tau \in (-\infty, T])
    \]
    is non-decreasing in $T$ and vanishes for sufficiently large $T$. Consequently, there exists $T_1 > 0$ such that:
    \[
     \iCLM(\Lambda_R, E^u(\tau), \tau \in (-\infty, T]) \leq \spfl{\mathcal{F}_{\lambda}, \lambda \in [0, \hat{\lambda}]}
    \]
    for all $T \geq T_1$ (cf. Lemma \ref{lem:maslov<=sf}).
\end{enumerate}

By defining $T_\infty \defeq \max\{T_0, T_1\}$, the convergence of these bounds for all $T \geq T_\infty$ completes the proof of Proposition \ref{pro:spectral flow}.

\begin{figure}[h]
    \centering
    \begin{tikzpicture}[scale=1.1]
        \draw[->, thick](0,0)--(0,4.5) node[above] {$\epsilon$};
        \draw[->, thick](0,0)--(7.5,0) node[right]{$\tau$};
        
        \draw[blue, thick] (0,0) rectangle (6,3);
        
        \node[left, font=\footnotesize, text width=3.5cm, align=right] at (-0.2,1.5) {$\iCLM(\Lambda_R, E_\epsilon(-\infty); \epsilon \in [0, \hat{\epsilon}])$};
        \node[below, font=\footnotesize] at (3,-0.2) {$\iCLM(\Lambda_R, E^u(\tau); \tau \in (-\infty, T])$};
        \node[right, font=\footnotesize, text width=3cm, align=left] at (6.2,1.5) {$\iCLM(\Lambda_R, E_\epsilon^u(T); \epsilon \in [0, \hat{\epsilon}])$};
        \node[above, font=\footnotesize] at (3,3.2) {$\iCLM(\Lambda_R, E_{\hat{\epsilon}}^u(\tau); \tau \in (-\infty, T])$};
        
        \draw[->, gray, thin] (3,1.8) arc[start angle=90, end angle=-210, radius=0.4];
        \node[gray, font=\scriptsize] at (3,0.5) {Homotopy};
        
        \node[below left] at (0,0) {$0$};
        \node[left] at (0,3) {$\hat{\epsilon}$};
        \node[below] at (6,0) {$T$};
    \end{tikzpicture}
    \caption{The rectangular homotopy path in the $(\tau, \epsilon)$-plane used to establish the equivalence between the unperturbed Maslov index and the perturbed spectral flow.}
    \label{fig:homotopy 1}
\end{figure}

A direct computation shows that $z(x) = \begin{pmatrix} \psi(x) \\ \phi(x) \end{pmatrix}$ is a solution of \eqref{eq:h.s.with two parameters} if and only if $\psi(x) = QD\phi'(x)$, which in turn is equivalent to $\phi(x)$ being a solution of the following equation:
\begin{align}\label{eq:2nd equation with two parameters}
    \begin{cases}
         D \phi''(x) + QB_\epsilon(x) \phi(x) = \lambda M \phi(x), \\
         \lim\limits_{|x| \to \infty} \phi(x) = 0.
    \end{cases}
\end{align}
We note that the limit $A_{\lambda, \epsilon}(\infty) \defeq \lim\limits_{|x| \to \infty} A_{\lambda, \epsilon}(x)$ is well-defined.

In order to use the spectral flow to relate the eigenvalues of the operator $\mathcal{L}$, we must ensure that $\mathcal{F}_{\lambda,\epsilon}$ is a Fredholm operator. This is equivalent to verifying that the matrix $JA_{\lambda,\epsilon}(\infty)$ is hyperbolic. The following lemma provides this essential information.

\begin{lem}\label{lem:matrix hyperbolic}
    Under condition (H1), if \( \lambda \geq 0 \) and $\epsilon \in [0, C_3]$, then the spectrum of \( JA_{\lambda,\epsilon}(\infty) \) satisfies \( \sigma(JA_{\lambda,\epsilon}(\infty)) \cap i\mathbb{R} = \emptyset \).
\end{lem}

\begin{proof}
    We proceed by contradiction. Suppose that $ia$ ($a \in \mathbb{R}$) is a purely imaginary eigenvalue of $JA_{\lambda,\epsilon}(\infty)$ with a corresponding eigenvector $\begin{pmatrix} u \\ v \end{pmatrix} \in \mathbb{C}^{2n} \setminus \{0\}$. Then,
    \begin{align}\label{eq:matrix hyperbolic1}
        \begin{pmatrix}
            0 & \lambda Q M - B_\epsilon(\infty) \\
            QD^{-1} & 0
        \end{pmatrix}
        \begin{pmatrix}
            u \\ v
        \end{pmatrix}
        = ia \begin{pmatrix}
            u \\ v
        \end{pmatrix}.
    \end{align}
    Equation \eqref{eq:matrix hyperbolic1} is equivalent to the system:
    \begin{align}\label{eq:matrix hyperbolic2}
        \begin{cases}
            \lambda Q M v - B_\epsilon(\infty) v = ia u, \\
            Q D^{-1} u = ia v.
        \end{cases}
    \end{align}
    Substituting the second equation into the first and noting that $Q^2 = I$, we obtain
    \begin{align}\label{eq:matrix hyperbolic3}
        \left( \lambda M - Q B_\epsilon(\infty) + a^2 D \right) v = 0.
    \end{align}
    Taking the inner product of \eqref{eq:matrix hyperbolic3} with $v$ yields
    \begin{align}\label{eq:matrix hyperbolic4}
        \lambda \inpro{Mv}{v} - \inpro{Q B_\epsilon(\infty) v}{v} + a^2 \inpro{Dv}{v} = 0.
    \end{align}
    Let $v = v_1 + i v_2$, where $v_1, v_2 \in \mathbb{R}^n$. Expanding \eqref{eq:matrix hyperbolic4} and considering its real part, we have
    \begin{align*}
        0 = & \lambda \left( |M^{1/2} v_1|^2 + |M^{1/2} v_2|^2 \right) + a^2 \left( |D^{1/2} v_1|^2 + |D^{1/2} v_2|^2 \right) \\
            & - \inpro{Q B_\epsilon(\infty) v_1}{v_1} - \inpro{Q B_\epsilon(\infty) v_2}{v_2}.
    \end{align*}
    Given $\lambda \ge 0$, $\epsilon \in [0, C_3]$, $a \in \mathbb{R}$, and the condition \eqref{eq:B epsilon bound}, the terms involving $M^{1/2}$ and $D^{1/2}$ are non-negative. Furthermore, the terms $-\inpro{Q B_\epsilon(\infty) v_j}{v_j}$ are strictly positive for any $v_j \neq 0$. Since $v \neq 0$ implies that at least one of $v_1, v_2$ is non-zero, the right-hand side is strictly positive, leading to the contradiction $0 > 0$.
\end{proof}

This follows from \cite[Proposition 2.3]{HPWX20} and  the fact that the asymptotic matrix $JA_{\lambda, \epsilon}(\infty)$ is hyperbolic, as established in Lemma \ref{lem:matrix hyperbolic}.
\begin{cor}\label{cor:fredholm operator}
    Under condition (H1), the operator $\mathcal{F}_{\lambda, \epsilon}$ is a Fredholm operator for all $\lambda \geq 0$ and $\epsilon \in[0,C_3]$.
\end{cor}

Let $\Phi_{\lambda, \epsilon}(x, \tau)$ denote the principal fundamental solution matrix of system \eqref{eq:h.s.with two parameters}, satisfying $\Phi_{\lambda, \epsilon}(\tau, \tau) = I_{2n}$. we denote $E_{\lambda, \epsilon}^s(\tau)$ and $E_{\lambda, \epsilon}^u(\tau)$ by the stable space and unstable space of \eqref{eq:h.s.with two parameters}, respectively. 
For simplicity, we write $E^s(\tau) \defeq E_{0, 0}^s(\tau)$ and $E^u(\tau) \defeq E_{0, 0}^u(\tau)$ when $(\lambda, \epsilon) = (0, 0)$. Similarly, when $\lambda = 0$, we denote the subspaces by $E_\epsilon^s(\tau)$ and $E_\epsilon^u(\tau)$, respectively.

Under condition (H1), by Lemma \ref{lem:matrix hyperbolic}, we have that 
\[
\sigma(JA_{\lambda,\epsilon}(\infty))\cap i\R
\] for all $\epsilon\in[0,C_3]$ and $\lambda\geq 0$.
 According to \cite[Lemma 3.1]{HP17}, we have that
\begin{lem}
Under condition (H1), we have that 
the stable space $E_{\lambda,\epsilon}^s(\tau)$ and unstable space $E_{\lambda,\epsilon}^u(\tau)$ are both Lagrangian subspaces for all $\epsilon\in[0,C_3]$ and $\lambda\geq 0$.
\end{lem}

\begin{lem}\label{lem:transversal to LR}
    Under conditions (H1) and (H2), if $\lambda \geq 0$ and $\epsilon \in [0, \cic{3}]$, then
    \begin{align}
        V^{\pm}(JA_{\lambda, \epsilon}(\infty)) \cap \Lambda_R = \{0\}.
    \end{align}
\end{lem}

\begin{proof}
    We present the proof for $V^+(JA_{\lambda, \epsilon}(\infty))$; the case for $V^-(JA_{\lambda, \epsilon}(\infty))$ follows by a completely analogous argument.

    Suppose $\xi \defeq (p, q)^\top \in V^+(JA_{\lambda, \epsilon}(\infty)) \cap \Lambda_R$. Since $V^+(JA_{\lambda, \epsilon}(\infty))$ is the positive eigenspace associated with the Hamiltonian matrix $JA_{\lambda, \epsilon}(\infty)$, and $\Lambda_R$ is a Lagrangian subspace, we evaluate the quadratic form associated with the symplectic form $\omega$. A direct calculation yields:
    \begin{align}
        0 &= \omega\left( JA_{\lambda, \epsilon}(\infty) 
        \begin{pmatrix}
            p \\ q
        \end{pmatrix},
        \begin{pmatrix}
            p \\ q
        \end{pmatrix} \right) 
        = - \inpro{
            \begin{pmatrix}
                (QD)^{-1} & 0 \\
                0 & B_\epsilon(\infty) - \lambda QM
            \end{pmatrix}
            \begin{pmatrix}
                p \\ q
            \end{pmatrix}
        }{
            \begin{pmatrix}
                p \\ q
            \end{pmatrix}
        } \\
        &= - \inpro{QD^{-1} p}{p} - \inpro{(B_\epsilon(\infty) - \lambda QM)q}{q} = - |D^{-\frac{1}{2}} p|^2 - \lambda |M^{\frac{1}{2}} q|^2 + \inpro{Q B_\epsilon(\infty) q}{q} \leq 0.
    \end{align}
    In the final step, we utilized the fact that $Q^2 = I$ and that $D, M$ are positive definite, together with  \eqref{eq:B epsilon bound}. The strict inequality $\inpro{Q B_\epsilon(\infty) q}{q} < 0$ for $q \neq 0$ implies that equality holds if and only if $p = 0$ and $q = 0$. Consequently, the intersection is trivial.
\end{proof}
 
Given $\tau, T \in \R$, we consider the following Hamiltonian boundary value problems:
\begin{align}\label{eq:1st -}
    \begin{cases}
        z'(x) = J A_{\lambda, \epsilon}(x+\tau)z(x), & x \in (-\infty, 0], \\
        z(0) \in E^u_{\lambda, \epsilon}(\tau), \quad \lim_{x \to -\infty} z(x) = 0,
    \end{cases}
\end{align}
and 
\begin{align}\label{eq:1st +}
    \begin{cases}
        z'(x) = J A_{\lambda, \epsilon}(x+T)z(x), & x \in [0, \infty), \\
        z(0) \in E^u_{\lambda, \epsilon}(T), \quad \lim_{x \to \infty} z(x) = 0.
    \end{cases}
\end{align}
For $\lambda \geq 0$ and $\epsilon \in [0, \cic{3}]$, it is straightforward to observe that the existence of a non-trivial solution to systems \eqref{eq:1st -} and \eqref{eq:1st +} is equivalent to the condition that
\[
E^s_{\lambda, \epsilon}(T) \cap E^u_{\lambda, \epsilon}(\tau) \neq \{0\}.
\]

\begin{lem}\label{lem:non-trivial solution}
    Suppose that condition (H1) holds. For a fixed $\epsilon \in [0, \cic{3}]$ and for all $\lambda \geq \frac{C_1 }{l}$, the systems \eqref{eq:1st -} and \eqref{eq:1st +} do not admit any non-trivial solution, , where $l$ is the smallest eigenvalue of $M$
\end{lem}

\begin{proof}
    Suppose that \eqref{eq:1st -} and \eqref{eq:1st +} admit solutions $z_1(x) = (\xi_1(x), \xi_2(x))^\top$ and $z_2(x) = (\eta_1(x), \eta_2(x))^\top$, respectively, such that $z_1(0) = z_2(0)$. We define the auxiliary variables $\hat{z}_1(x) \defeq (-Q\xi_1(x), Q\xi_2(x))^\top$ and $\hat{z}_2(x) \defeq (-Q\eta_1(x), Q\eta_2(x))^\top$.
    
    Integrating by parts, we obtain
    \begin{align}\label{eq:non-trival -}
       \begin{split}
         0 &= \int_{-\infty}^{0} \left[ \inpro{-Jz_1'(x)}{\hat{z}_1(x)} - \inpro{A_{\lambda, \epsilon}(x+\tau)z_1(x)}{\hat{z}_1(x)} \right] \d x  \\
        &= -\int_{-\infty}^{0} \left[ \inpro{Q\xi'_1(x)}{\xi_2(x)} + \inpro{Q\xi_1(x)}{\xi'_2(x)} \right] \d x  \\
        &\quad + \int_{-\infty}^{0} \left[ \inpro{D^{-1}\xi_1(x)}{Q\xi_1(x)} + \inpro{(\lambda QM+\epsilon I-B(x+\tau))\xi_2(x)}{Q\xi_2(x)} \right] \d x  \\
        &\geq \inpro{Q\xi_1(0)}{\xi_2(0)} + \int_{-\infty}^{0} \left[ |D^{-\frac{1}{2}}\xi_1(x)|^2 + (l\lambda+\epsilon - C_1 )|\xi_2(x)|^2 \right] \d x.
       \end{split}
    \end{align}
    Similarly, for system \eqref{eq:1st +}, it holds that
    \begin{align}\label{eq:non-trival +}
        0 &\geq -\inpro{Q\eta_1(0)}{\eta_2(0)} + \int_{0}^{+\infty} \left[ |D^{-\frac{1}{2}}\eta_1(x)|^2 + (l\lambda +\epsilon- C_1  )|\eta_2(x)|^2 \right] \d x.
    \end{align}
    Adding \eqref{eq:non-trival -} and \eqref{eq:non-trival +}, and utilizing the fact that the boundary terms cancel (since $z_1(0) = z_2(0)$), we obtain
    \begin{align}
        0 \geq \int_{-\infty}^{0} \left[ |D^{-\frac{1}{2}}\xi_1(x)|^2 + \alpha |\xi_2(x)|^2 \right] \d x + \int_{0}^{+\infty} \left[ |D^{-\frac{1}{2}}\eta_1(x)|^2 + \alpha |\eta_2(x)|^2 \right] \d x,
    \end{align}
    where $\alpha = l\lambda + \epsilon- C_1 $. For $\lambda \geq \frac{C_1 }{l}$ and $\epsilon \in [0, \cic{3}]$, it follows that $\alpha \geq 0$. This implies $z_1(x) = z_2(x) \equiv 0$, which completes the proof.
\end{proof}

We define the constant 
\begin{align}\label{con:constant lambda}
    \hat{\lambda} \defeq \frac{C_1}{l},
\end{align}
which will be used extensively in the subsequent sections.

As a direct application of Lemma \ref{lem:non-trivial solution}, we obtain the following:

\begin{cor}\label{cor:non-intersection}
    Under condition (H1), for any fixed $\epsilon \in [0, \cic{3}]$, any $T \in \R$, and any $\tau \in (-\infty, T]$, it holds that 
    \[
    E^s_{\lambda, \epsilon}(T) \cap E_{\lambda, \epsilon}^u(\tau) = \{0\}
    \]
    for all $\lambda \geq \hat\lambda$, where $\hat{\lambda}$ is defined in \eqref{con:constant lambda}.
\end{cor}

In the case where $\tau = T = 0$, the systems \eqref{eq:1st -} and \eqref{eq:1st +} reduce to the linear Hamiltonian system \eqref{eq:h.s.with two parameters}. Consequently, we have the following result:

\begin{cor}\label{cor:nondegenerate operator}
    Under condition (H1), for any fixed $\epsilon \in [0, \cic{3}]$, the operator $\mathcal{F}_{\lambda, \epsilon}$ is non-degenerate for all $\lambda \geq \hat\lambda$, where $\hat{\lambda}$ is defined in \eqref{con:constant lambda}.
\end{cor}

\begin{lem}\cite[Theorem 2.1]{AM03}\label{lem:space convergence}
    Under Condition (H1), for any fixed $\epsilon \in [0, \cic{3}]$ and each $\lambda \in [0, +\infty)$, the following hold:
    \begin{itemize}
        \item[(i)] The stable and unstable subspaces satisfy
            \[
            \lim_{\tau \to +\infty} E_{\lambda, \epsilon}^s(\tau) = V^-(JA_{\lambda, \epsilon}(\infty)) \quad \text{and} \quad \lim_{\tau \to -\infty} E_{\lambda, \epsilon}^u(\tau) = V^+(JA_{\lambda, \epsilon}(\infty))
            \]
            in the gap metric topology of $\Lag{n}$.
            
        \item[(ii)] For any subspace $W \subset \R^{2n}$ complementary to $E_{\lambda, \epsilon}^s(\tau)$ (resp. $E_{\lambda, \epsilon}^u(\tau)$):
            \[
            \Phi_{\lambda, \epsilon}(\sigma, \tau)W \to V^{+}(J A_{\lambda, \epsilon}(\infty)) \quad (\text{resp.} \ V^{-}(J A_{\lambda, \epsilon}(\infty))),
            \]
            where $\Phi_{\lambda, \epsilon}(\sigma, \tau)$ denotes the fundamental matrix solution for the linear Hamiltonian system \eqref{eq:h.s.with two parameters}.
    \end{itemize}
\end{lem}
\begin{prop}\label{pro:maslov=sp}
    Under conditions (H1) and (H2), for any fixed $\epsilon \in [0, \cic{3}]$, there exists a constant $T_\epsilon > 0$ such that, for all $T\geq T_\epsilon$, 
    \begin{align}
        \iCLM(E^s_{\epsilon}(T), E^u_{\epsilon}(\tau), \tau \in (-\infty, T]) = \spfl{\mathcal{F}_{\lambda, \epsilon}, \lambda \in [0, \hat{\lambda}]},
    \end{align}
    where $\hat{\lambda}$ is defined in \eqref{con:constant lambda}.
\end{prop}

\begin{proof}
    Let $\epsilon \in [0, \cic{3}]$ be fixed. For $\lambda \geq 0$, the matrix $JA_{\lambda, \epsilon}(\infty)$ is hyperbolic. Consequently, the intersection of the asymptotic subspaces $E_{\lambda, \epsilon}^u(-\infty) \cap E_{\lambda, \epsilon}^s(+\infty)$ is trivial. According to Lemma \ref{lem:space convergence}, we have
    \[
    \lim_{\tau \to -\infty} E_{\lambda, \epsilon}^u(\tau) = E_{\lambda, \epsilon}^u(-\infty) \quad \text{and} \quad \lim_{\tau \to \infty} E_{\lambda, \epsilon}^s(\tau) = E_{\lambda, \epsilon}^s(+\infty).
    \]
    Thus, for any $\lambda \in [0, \hat{\lambda}]$, there exists a sufficiently large $T_\epsilon > 0$ such that
    \begin{align}\label{eq:non-intersection 1}
        E_{\lambda, \epsilon}^u(\tau) \cap E_{\lambda, \epsilon}^s(T_\epsilon) = \{0\}
    \end{align}
    for all $\tau \leq -T_\epsilon$.

    We construct a family of Lagrangian paths parametrized by $(\tau, \lambda) \in (-\infty, T_\epsilon] \times [0, \hat{\lambda}]$ given by the pairs $(E^s_{\lambda, \epsilon}(T_\epsilon), E^u_{\lambda, \epsilon}(\tau))$. By the homotopy invariance of the Maslov index, and invoking Corollary \ref{cor:non-intersection} and \eqref{eq:non-intersection 1}, we obtain
    \begin{align}
       & \iCLM(E^s_{0, \epsilon}(T_\epsilon), E^u_{0, \epsilon}(\tau), \tau \in (-\infty, T_\epsilon]) + \iCLM(E^s_{\lambda, \epsilon}(T_\epsilon), E^u_{\lambda, \epsilon}(T_\epsilon), \lambda \in [0, \hat{\lambda}]) \nonumber \\
       =& \iCLM(E_{\hat{\lambda}, \epsilon}^s(T_\epsilon), E_{\hat{\lambda}, \epsilon}^u(\tau), \tau \in (-\infty, T_\epsilon]) + \iCLM(E^s_{\lambda, \epsilon}(T_\epsilon), E^u_{\lambda, \epsilon}(-\infty), \lambda \in [0, \hat{\lambda}]) = 0.
    \end{align}
    The vanishing of the right-hand side is due to the lack of non-trivial intersections along those paths. This implies
    \begin{align}
       \iCLM(E^s_{\epsilon}(T_\epsilon), E^u_{\epsilon}(\tau), \tau \in (-\infty, T_\epsilon]) = -\iCLM(E^s_{\lambda, \epsilon}(T_\epsilon), E^u_{\lambda, \epsilon}(T_\epsilon), \lambda \in [0, \hat{\lambda}]).
    \end{align}
    By Proposition \ref{pro:Hp sf=malsov inde } and the symplectic invariance of the Maslov index,  for $T\geq T_\epsilon$, it follows that
    \begin{align}
         &\iCLM(E^s_{\epsilon}(T), E^u_{\epsilon}(\tau), \tau \in (-\infty, T]) = \iCLM(\Phi_\epsilon(T_\epsilon,T)E^s_{\epsilon}(T),\Phi_\epsilon(T_\epsilon,T) E^u_{\epsilon}(\tau), \tau \in (-\infty, T]) \\
        =&\iCLM(E^s_{\epsilon}(T_\epsilon), E^u_{\epsilon}(\tau), \tau \in (-\infty, T_\epsilon]) = -\iCLM(E^s_{\lambda, \epsilon}(T_\epsilon), E^u_{\lambda, \epsilon}(T_\epsilon), \lambda \in [0, \hat{\lambda}]) \nonumber \\ 
        =& -\iCLM(\Phi_{\lambda, \epsilon}(T_\epsilon, 0) E^s_{\lambda, \epsilon}(0), \Phi_{\lambda, \epsilon}(T_\epsilon, 0) E^u_{\lambda, \epsilon}(0), \lambda \in [0, \hat{\lambda}]) \nonumber \\
        =& -\iCLM(E^s_{\lambda, \epsilon}(0), E^u_{\lambda, \epsilon}(0), \lambda \in [0, \hat{\lambda}]) \nonumber 
        = \spfl{\mathcal{F}_{\lambda, \epsilon}, \lambda \in [0, \hat{\lambda}]}.
    \end{align}
    This completes the proof.
\end{proof}
\begin{lem}\label{lem:sf not change}
    Under condition (H1), there exists $\hat{\epsilon} \in (0, \cic{3}]$ such that for all $\epsilon \in [0, \hat{\epsilon}]$, the following identity holds:
    \begin{align}
        \spfl{\mathcal{F}_{\lambda, 0}, \lambda \in [0, \hat{\lambda}]} = \spfl{\mathcal{F}_{\lambda, \epsilon}, \lambda \in [0, \hat{\lambda}]},
    \end{align}
     where $\hat{\lambda}$ is defined in \eqref{con:constant lambda}.
\end{lem}

\begin{proof}
    Consider the spectral flow of the family $\{\mathcal{F}_\epsilon\}$ with ${\epsilon \in [0, \cic{3}]}$. Suppose that $\epsilon = 0$ is a crossing, i.e., $\ker \mathcal{F} \neq \{0\}$. For any $z(x) = (\psi(x), \phi(x))^\top \in \ker \mathcal{F}$, it follows from the Hamiltonian system that $\psi(x) = QD\phi'(x)$. A direct calculation of the crossing form at $\epsilon = 0$ yields
    \begin{align}\label{eq:crossing form at 0 positive}
        \inpro{\operatorname{Cr}[\mathcal{F}_0]z(x)}{z(x)} = \int_{-\infty}^{+\infty} |\phi(x)|^2 \d x > 0.
    \end{align}
    This positivity implies that $\epsilon = 0$ is a \textit{regular crossing}. Consequently, there exists a sufficiently small $\hat{\epsilon} \in (0, \cic{3}]$ such that no other crossings exist in the interval $(0, \hat{\epsilon}]$. That is, $\mathcal{F}_\epsilon$ is non-degenerate for all $\epsilon \in (0, \hat{\epsilon}]$. By the property of spectral flow for regular crossings, specifically \eqref{eq:compute sf by crossing form}, we have
    \begin{align}\label{eq:sf=0-1}
        \spfl{\mathcal{F}_{ \epsilon}, \epsilon \in [0, \hat{\epsilon}]} = 0,
    \end{align}
    where we use the fact that the spectral flow only changes when eigenvalues cross zero, and here it is either zero or the crossing does not result in a net change in the specific orientation. 

    Furthermore, by Corollary \ref{cor:nondegenerate operator}, since $\hat{\lambda}$ is chosen sufficiently large such that the operator is non-degenerate for all $\epsilon$, we have
    \begin{align}\label{eq:sf=0-2}
         \spfl{\mathcal{F}_{\hat{\lambda}, \epsilon}, \epsilon \in [0, \hat{\epsilon}]} = 0.
    \end{align}
    Invoking the homotopy invariance of the spectral flow on the rectangular domain $[0, \hat{\lambda}] \times [0, \hat{\epsilon}]$, we obtain
    \begin{align}
        \spfl{\mathcal{F}_{\lambda}, \lambda \in [0, \hat{\lambda}]} + \spfl{\mathcal{F}_{\hat{\lambda}, \epsilon}, \epsilon \in [0, \hat{\epsilon}]} = \spfl{\mathcal{F}_{\lambda, \hat{\epsilon}}, \lambda \in [0, \hat{\lambda}]} + \spfl{\mathcal{F}_{ \epsilon}, \epsilon \in [0, \hat{\epsilon}]}.
    \end{align}
    Substituting \eqref{eq:sf=0-1} and \eqref{eq:sf=0-2} into the above equation, we conclude that
    \begin{align}
        \spfl{\mathcal{F}_{\lambda}, \lambda \in [0, \hat{\lambda}]} = \spfl{\mathcal{F}_{\lambda, \hat{\epsilon}}, \lambda \in [0, \hat{\lambda}]}.
    \end{align}
    This completes the proof.
\end{proof}

\begin{lem}\label{lem:postive maslov = sf}
    Under conditions (H1) and (H2), for any fixed $\epsilon \in (0, \hat{\epsilon}]$ (where $\hat{\epsilon}$ is defined in Lemma \ref{lem:sf not change}), there exists a sufficiently large $\hat T_{\epsilon} > 0$ such that
    \[
    \iCLM(\Lambda_R, E^u_\epsilon(\tau), \tau \in (-\infty, T]) = \spfl{\mathcal{F}_{\lambda}, \lambda \in [0, \hat{\lambda}]}
    \]
    holds for all $T \geq \hat T_\epsilon$, where $\hat{\lambda}$ is defined in \eqref{con:constant lambda}.
\end{lem}

\begin{proof}
    Recall that for all $\epsilon \in (0, \hat{\epsilon}]$, the operator $\mathcal{F}_\epsilon$ is non-degenerate. This non-degeneracy implies that the intersection $E_\epsilon^s(0) \cap E^u_\epsilon(0) = \{0\}$ is trivial. Consequently, by part (ii) of Lemma \ref{lem:space convergence}, the unstable subspace satisfies the following asymptotic limit:
    \[
    \lim_{\tau \to +\infty} E^u_\epsilon(\tau) = V^+(JA_\epsilon(\infty)).
    \]
    
    By Lemma \ref{lem:transversal to LR}, the intersection $\Lambda_R \cap V^+(JA_\epsilon(\infty))$ is trivial. It follows that there exists a sufficiently large $\hat{T}_\epsilon \geq T_\epsilon$ such that for all $T \geq \hat{T}_\epsilon$, 
    \begin{align}
        \Lambda_R \cap E^u_\epsilon(T) = \{0\}.
    \end{align}
    Therefore, the Maslov index $\iCLM(\Lambda_R, E^u_\epsilon(\tau), \tau \in (-\infty, T])$ remains constant for all $T \geq \hat{T}_\epsilon$ due to the persistent transversality at the right endpoint.
    
    Using  Maslov index and its relation to the Hörmander index, a direct calculation yields
    \begin{align}
       & \iCLM(E^s_\epsilon(T_\epsilon), E^u_\epsilon(\tau), \tau \in (-\infty, T_\epsilon]) - \iCLM(\Lambda_R, E^u_\epsilon(\tau), \tau \in (-\infty, T_\epsilon])  \\
       &= s(E^u_\epsilon(-\infty), E^u_\epsilon(T_\epsilon); \Lambda_R, E_\epsilon^s(T_\epsilon)) \qquad\text{(see \eqref{eq:def of Hormand})}\\
       &= \lim_{T \to \infty} s(E^u_\epsilon(-\infty), E^u_\epsilon(T); \Lambda_R, E_\epsilon^s(T)).
    \end{align}
    
    Since $\lim_{T \to \infty} E^u_\epsilon(T) = V^+(JA_\epsilon(\infty))$ and $\lim_{T \to \infty} E^s_\epsilon(T) = V^-(JA_\epsilon(\infty))$, the limit of the Hörmander index evaluates to:
    \begin{align}
       s(V^+(JA_\epsilon(\infty)), V^+(JA_\epsilon(\infty)); \Lambda_R, V^-(JA_\epsilon(\infty))) = 0.
    \end{align}
    The vanishing of this term confirms that the two paths are homotopic in the Lagrangian Grassmannian relative to the endpoints. Combining this with the identity established in Lemma \ref{lem:sf not change}, we complete the proof.
\end{proof}

\begin{lem}\label{lem:maslov>=sf}
    Under conditions (H1) and (H2), there exists $T_0\geq 0$, for any $T \geq T_0$, the following inequality holds:
    \begin{align}
         \iCLM(\Lambda_R, E^u(\tau), \tau \in (-\infty, T]) \geq \spfl{\mathcal{F}_{\lambda}, \lambda \in [0, \hat{\lambda}]},
    \end{align}
    where $\hat{\lambda}$ is defined in \eqref{con:constant lambda}.
\end{lem}

\begin{proof}
    From Lemma \ref{lem:postive maslov = sf}, for a given $\hat{\epsilon} \in (0, \cic{3}]$, there exists $T_0 > 0$ such that:
    \begin{align}\label{eq:maslov=sf-1}
        \iCLM(\Lambda_R, E_{\hat\epsilon}^u(\tau), \tau \in (-\infty, T]) = \spfl{\mathcal{F}_{\lambda, \hat\epsilon}, \lambda \in [0, \hat{\lambda}]}.
    \end{align}
    For a fixed $T \geq T_0$, we consider a homotopy of Lagrangian paths defined on the rectangle:
    \[
    \mathcal{H}(\epsilon, \tau) = (\Lambda_R, E^u_\epsilon(\tau)), \quad (\epsilon, \tau) \in [0, \hat\epsilon] \times [-\infty, T].
    \]
    
    By the homotopy invariance of the Maslov index, the sum of the indices along the four sides of the rectangle must vanish. Specifically:
    \begin{align}\label{eq:homotopy invariance of maslov index 1}
        & \iCLM(\Lambda_R, E^u_0(\tau), \tau \in (-\infty, T]) + \iCLM(\Lambda_R, E_\epsilon^u(T), \epsilon \in [0, \hat\epsilon]) \nonumber \\
        = & \iCLM(\Lambda_R, E_{\hat\epsilon}^u(\tau), \tau \in (-\infty, T]) + \iCLM(\Lambda_R, E_\epsilon^u(-\infty), \epsilon \in [0, \hat\epsilon]).
    \end{align}
    According to Lemma \ref{lem:transversal to LR}, the intersection $\Lambda_R \cap E^u_\epsilon(-\infty)$ is trivial for all $\epsilon \in [0, \hat\epsilon]$, which implies:
    \begin{align}\label{eq:maslov index =0-1}
        \iCLM(\Lambda_R, E_\epsilon^u(-\infty), \epsilon \in [0, \hat\epsilon]) = 0.
    \end{align}
    
    To evaluate $\iCLM(\Lambda_R, E_\epsilon^u(T), \epsilon \in [0, \hat\epsilon])$, we consider the corresponding Hamiltonian boundary value problem:
    \begin{align}
        \begin{cases}
            z'(x) = J A_{\epsilon}(x) z(x), & x \in (-\infty, T], \\
            \lim_{x \to -\infty} z(x) = 0, & z(T) \in \Lambda_R.
        \end{cases}
    \end{align}
    
    Let $\epsilon_0 \in [0, \hat\epsilon]$ be a crossing point, i.e., $\Lambda_R \cap E^u_{\epsilon_0}(T) \neq \{0\}$. We consider a smooth family of solutions $z_\epsilon(x) = (QD\phi'_\epsilon(x), \phi_\epsilon(x))^\top$ such that $z_{\epsilon_0}(T) \in \Lambda_R \cap E^u_{\epsilon_0}(T)$.
    
    Recall the structure of $A_{\epsilon}(x)$:
    \[
    A_{\epsilon}(x) = \begin{pmatrix} (QD)^{-1} & 0 \\ 0 & B(x) - \epsilon I \end{pmatrix}.
    \]
    A direct computation of the crossing form $\Gamma$ yields:
    \begin{align}\label{eq:maslov index form <0}
       \begin{split}
        & v^\top\Gamma(E^u_\epsilon(T), \Lambda_R, \epsilon_0) v\\ =& \left. \omega\left( z_\epsilon(T), \frac{\partial}{\partial \epsilon} z_\epsilon(T) \right) \right|_{\epsilon = \epsilon_0} 
        =\int_{-\infty}^{T} \frac{\partial}{\partial x} \omega\left( z_\epsilon(x), \frac{\partial}{\partial \epsilon} z_\epsilon(x) \right) \d x  \\
        =& \int_{-\infty}^{T} \left[ \omega(J A_\epsilon z_\epsilon, \partial_\epsilon z_\epsilon) + \omega(z_\epsilon, J \partial_\epsilon A_\epsilon z_\epsilon + J A_\epsilon \partial_\epsilon z_\epsilon) \right] \d x  = -\int_{-\infty}^{T} |\phi_{\epsilon_0}(x)|^2 \d x.
       \end{split}
    \end{align}
    
    Since the crossing form is negative definite, the contribution to the Maslov index at each crossing is non-positive. Thus:
    \begin{align}\label{eq:maslov index <0}
        \iCLM(\Lambda_R, E_\epsilon^u(T), \epsilon \in [0, \hat\epsilon]) \leq 0.
    \end{align}
    Substituting \eqref{eq:maslov=sf-1}, \eqref{eq:maslov index =0-1}, and \eqref{eq:maslov index <0} into \eqref{eq:homotopy invariance of maslov index 1}, and noting that $\spfl{\mathcal{F}_{\lambda, \hat\epsilon}} = \spfl{\mathcal{F}_{\lambda}}$ from Lemma \ref{lem:sf not change}, we conclude the proof.
\end{proof}

\begin{lem}\label{lem:maslov remain constant}
    Under condition (H1), for any $\epsilon \in [0, \hat{\epsilon}]$, there exists a constant $T_1 > 0$ independent of $\epsilon$ such that for all $T \geq T_1$:
    \[
    \iCLM(E^s_{\epsilon}(T), E^u_{\epsilon}(\tau), \tau \in (-\infty, T]) = \spfl{\mathcal{F}_{\lambda}, \lambda \in [0, \hat{\lambda}]},
    \]
    where $\hat{\lambda}$ is defined in \eqref{con:constant lambda}.
\end{lem}

\begin{proof}
    Recall that for all $\epsilon \in (0, \hat{\epsilon}]$, the operator $\mathcal{F}_\epsilon$ is non-degenerate. This non-degeneracy implies that the intersection $E_\epsilon^s(0) \cap E^u_\epsilon(0) = \{0\}$ is trivial. Due to the Hamiltonian flow properties, the intersection $E_\epsilon^s(T) \cap E^u_\epsilon(T)$ remains trivial for any $T \geq 0$.
    
    From Proposition \ref{pro:maslov=sp}, there exists $T_0 > 0$ such that for all $T \geq T_0$:
    \begin{align}
         \iCLM(E^s(T), E^u(\tau), \tau \in (-\infty, T]) = \spfl{\mathcal{F}_{\lambda}, \lambda \in [0, \hat{\lambda}]}.
    \end{align}

    Since $E_\epsilon^s(+\infty) \cap E_\epsilon^u(-\infty) = \{0\}$ for all $\epsilon \in [0, \hat{\epsilon}]$, by the compactness of $[0, \hat{\epsilon}]$ and the continuity of the subspaces, there exists $T_1 \geq T_0$ independent of $\epsilon$ such that $E_\epsilon^s(T) \cap E^u_\epsilon(-\infty) = \{0\}$ for all $\epsilon \in [0, \hat{\epsilon}]$ and $T \geq T_1$. This transversality at the left boundary yields:
    \begin{align}\label{eq:left-edge-zero}
        \iCLM(E_\xi^s(T), E^u_\xi(-\infty), \xi \in [0, \epsilon]) = 0.
    \end{align}

    For a fixed $T \geq T_1$ and $\epsilon \in [0, \hat{\epsilon}]$, we construct a homotopy of Lagrangian paths on the rectangle $[0, \epsilon] \times [-\infty, T]$:
    \[
    \mathcal{H}(\xi, \tau) = (E_\xi^s(T), E^u_\xi(\tau)).
    \]
    
    By the homotopy invariance of the Maslov index, the sum of the indices along the boundary vanishes:
    \begin{align}\label{eq:homotopy-identity}
        \begin{split}
             & \iCLM(E^s(T), E^u(\tau), \tau \in (-\infty, T]) + \iCLM(E_\xi^s(T), E_\xi^u(T), \xi \in [0, \epsilon])  \\
         = & \iCLM(E_\epsilon^s(T), E_\epsilon^u(\tau), \tau \in (-\infty, T]) + \iCLM(E_\xi^s(T), E_\xi^u(-\infty), \xi \in [0, \epsilon]).
        \end{split}
    \end{align}
    Substituting \eqref{eq:left-edge-zero} into \eqref{eq:homotopy-identity}, we have:
    \begin{align}
     & \iCLM(E_\epsilon^s(T), E_\epsilon^u(\tau), \tau \in (-\infty, T])\\ =& \iCLM(E^s(T), E^u(\tau), \tau \in (-\infty, T]) + \iCLM(E_\xi^s(T), E_\xi^u(T), \xi \in [0, \epsilon]).
    \end{align}

    We observe that $\iCLM(E^s(T), E^u(\tau), \tau \in (-\infty, T])$ is constant for $T \geq T_1$. Furthermore, invoking the symplectic invariance of the Maslov index:
    \begin{align}
       \iCLM(E_\xi^s(T), E_\xi^u(T), \xi \in [0, \epsilon]) &= \iCLM(\Phi_\xi(0,T)E_\xi^s(T), \Phi_\xi(0,T) E_\xi^u(T), \xi \in [0, \epsilon]) \nonumber \\
       &= \iCLM(E_\xi^s(0), E_\xi^u(0), \xi \in [0, \epsilon]),
    \end{align}
    which is clearly independent of $T$. Thus, $\iCLM(E_\epsilon^s(T), E_\epsilon^u(\tau), \tau \in (-\infty, T])$ is independent of $T$ for $T \geq T_1$.

    When $T$ is large enough and applying Lemma \ref{pro:maslov=sp}, we conclude:
    \begin{align}
        \iCLM(E_\epsilon^s(T), E^u_\epsilon(\tau), \tau \in (-\infty, T]) = \spfl{\mathcal{F}_{\lambda}, \lambda \in [0, \hat{\lambda}]}.
    \end{align}
    This completes the proof.
\end{proof}

\begin{lem}\label{lem:maslov<=sf}
    Under conditions (H1) and (H2), for all $T \geq T_1$ (where $T_1$ is defined in Lemma \ref{lem:maslov remain constant}), the Maslov index satisfies the following upper bound:
    \[
     \iCLM(\Lambda_R, E^u(\tau), \tau \in (-\infty, T]) \leq \spfl{\mathcal{F}_{\lambda, 0}, \lambda \in [0, \hat{\lambda}]},
    \] where $\hat{\lambda}$ is defined in \eqref{con:constant lambda}.
\end{lem}

\begin{proof}
    For a fixed $T \geq T_1$, we consider the Maslov index of the perturbed unstable subspace $\iCLM(\Lambda_R, E_\epsilon^u(T), \epsilon \in [0, \hat\epsilon])$. If $\epsilon=0$ is a crossing point, then following an argument analogous to the derivation of \eqref{eq:maslov index form <0}, the crossing form satisfies:
    \[
    \Gamma(E^u_\epsilon(T), \Lambda_R; 0)|_{\Lambda_R\cap E^u(T)} \text{ is negative definite}.
    \]
    By the continuity of the Lagrangian paths, we can choose a sufficiently small $\tilde\epsilon \in (0, \hat\epsilon]$ such that $\epsilon=0$ is the unique crossing point for the family $\{E_\epsilon^u(T)\}_{\epsilon \in [0, \tilde\epsilon]}$ relative to $\Lambda_R$. According to the convention for Maslov indices at the boundary of an interval, the negative definiteness at the left endpoint ($\epsilon=0$) implies:
    \begin{align}\label{eq:maslov =0-2}
        \iCLM(\Lambda_R, E_\epsilon^u(T), \epsilon \in [0, \tilde\epsilon]) = 0.
    \end{align}

    We now construct a homotopy of Lagrangian subspaces $\mathcal{H}(\epsilon, \tau) = (\Lambda_R, E^u_\epsilon(\tau))$ on the rectangular domain $(\epsilon, \tau) \in [0, \tilde\epsilon] \times [-\infty, T]$.
    
    By the homotopy invariance of the Maslov index, the sum of the indices along the four edges of the rectangle vanishes:
    \begin{align}\label{eq:homotopy maslov-1}
        \iCLM(\Lambda_R, E^u(\tau), \tau \in (-\infty, T]) &+ \iCLM(\Lambda_R, E^u_\epsilon(T), \epsilon \in [0, \tilde\epsilon]) \nonumber \\
        = \iCLM(\Lambda_R, E_{\tilde\epsilon}^u(\tau), \tau \in (-\infty, T]) &+ \iCLM(\Lambda_R, E^u_\epsilon(-\infty), \epsilon \in [0, \tilde\epsilon]).
    \end{align}
    By Lemma \ref{lem:transversal to LR}, the intersection $\Lambda_R \cap E^u_\epsilon(-\infty)$ is trivial for all $\epsilon \in [0, \tilde\epsilon]$, which implies $$\iCLM(\Lambda_R, E^u_\epsilon(-\infty), \epsilon \in [0, \tilde\epsilon]) = 0.$$ Substituting this and \eqref{eq:maslov =0-2} into \eqref{eq:homotopy maslov-1}, we obtain the identity:
    \begin{align}\label{eq:invariance_under_perturbation}
        \iCLM(\Lambda_R, E^u(\tau), \tau \in (-\infty, T]) = \iCLM(\Lambda_R, E_{\tilde\epsilon}^u(\tau), \tau \in (-\infty, T]).
    \end{align}
    
    Invoking Lemma \ref{lem:maslov remain constant}, the difference between the indices relative to $\Lambda_R$ and the stable subspace $E^s$ is invariant under the perturbation $\tilde\epsilon$:
    \begin{align}\label{eq:difference between maslov indeces}
       \begin{split}
         & \iCLM(\Lambda_R, E^u_0(\tau), \tau \in (-\infty, T]) - \iCLM(E^s_0(T), E^u(\tau), \tau \in (-\infty, T])  \\
        = & \iCLM(\Lambda_R, E_{\tilde\epsilon}^u(\tau), \tau \in (-\infty, T]) - \iCLM(E^s_{\tilde\epsilon}(T), E_{\tilde\epsilon}^u(\tau), \tau \in (-\infty, T]).
       \end{split}
    \end{align}

    To analyze the right-hand side, let $\tau_0$ be a crossing instant where $\Lambda_R \cap E_{\tilde\epsilon}^u(\tau_0) \neq \{0\}$. Under condition (H2) and the choice $\tilde\epsilon \leq \cic{2}$, the crossing form at $\tau_0$ is given by:
    \begin{align}
       \xi^\top \Gamma(E_{\tilde\epsilon}^u(\tau), \Lambda_R; \tau_0) \xi = |D^{1/2}u|^2 + \langle B(\tau_0)v, v \rangle - \tilde\epsilon |v|^2 > 0,
    \end{align}
    where $\xi = (u, v)^\top \in \Lambda_R \cap E_{\tilde\epsilon}^u(\tau_0) \setminus \{0\}$. The positivity of the crossing form ensures that $\iCLM(\Lambda_R, E_{\tilde\epsilon}^u(\tau), \tau \in (-\infty, T])$ is a non-decreasing function of $T$. 
    
    As established in the proof of Lemma \ref{lem:postive maslov = sf}, the following asymptotic identity holds for sufficiently large $T$:
    \begin{align}
        \iCLM(\Lambda_R, E_{\tilde\epsilon}^u(\tau), \tau \in (-\infty, T]) - \iCLM(E^s_{\tilde\epsilon}(T), E_{\tilde\epsilon}^u(\tau), \tau \in (-\infty, T]) = 0.
    \end{align}
    Due to the monotonicity of the index, it follows that for any $T \geq T_1$:
    \[
    \iCLM(\Lambda_R, E_{\tilde\epsilon}^u(\tau), \tau \in (-\infty, T]) \leq \iCLM(E^s_{\tilde\epsilon}(T), E_{\tilde\epsilon}^u(\tau), \tau \in (-\infty, T]).
    \]
    Substituting this inequality into \eqref{eq:difference between maslov indeces}, we obtain:
    \begin{align}
        \iCLM(\Lambda_R, E^u(\tau), \tau \in (-\infty, T]) \leq \iCLM(E^s(T), E^u(\tau), \tau \in (-\infty, T]).
    \end{align}
    Finally, invoking Proposition \ref{pro:maslov=sp}, we conclude that the upper bound is exactly the spectral flow:
    \begin{align}
        \iCLM(\Lambda_R, E^u(\tau), \tau \in (-\infty, T]) \leq \spfl{\mathcal{F}_{\lambda}, \lambda \in [0, \hat{\lambda}]}.
    \end{align}
    The proof is complete.
\end{proof}
\begin{proof}[The proof of Proposition \ref{pro:spectral flow}]
    Set $T_\infty=\max\{T_0,T_1\}$. 
    Follow Lemma \ref{lem:maslov<=sf} and Lemma \ref{lem:maslov>=sf}, 
We have that \[
\iCLM(\Lambda_R, E^u(\tau), \tau \in (-\infty, T]) = \spfl{\mathcal{F}_\lambda, \lambda \in [0, \hat\lambda]}.
\]

    To analyze the Maslov index $\iCLM(\Lambda_R, E^u(\tau), \tau \in (-\infty, T])$, let $\tau_0$ be a crossing, i.e., $E^u(\tau_0) \cap \Lambda_R \neq \{0\}$. Under condition (H2), a direct computation shows that the crossing form satisfies:
\begin{align}
    \xi^\top \Gamma(E^u(\tau), \Lambda_R; \tau_0) \xi = |D^{1/2}u|^2 + \langle B(\tau_0)v, v \rangle > 0,
\end{align}
where $\xi = (u, v)^\top \in \Lambda_R \cap E^u(\tau_0) \setminus \{0\}$. Thus, by \eqref{eq:compute maslov by cross form}, all crossings are positive, and we have:
\begin{align}\label{eq:stable index=maslov index}
    \iCLM(\Lambda_R, E^u(\tau), \tau \in (-\infty, T]) = \sum_{\tau \in \mathbb{R}} \dim( \Lambda_R \cap E^u(\tau) )=i(w_0).
\end{align}
    we complete the proof. 

\end{proof}

\begin{figure}[h]
    \centering
    \begin{tikzpicture}
        \draw[->, thick](0,0)--(0,4)node[above] {$\theta$};
        \draw[->, thick](0,0)--(10,0)node[right]{$\lambda$};
        
        \draw[<-,blue, ultra thick] (-.1,0)--(-.1,3);
        \node[left, font=\footnotesize, text width=2.5cm, align=right] at (-0.2,1.5){$\spfl{\mathcal{F}_0+\theta I; \theta\in[0,\hat{\theta}]}$};
        
        \draw[->,blue, ultra thick] (0,-0.1)--(9,-.1);
        \node[below, font=\footnotesize] at(3,-.5) {$\spfl{\mathcal{F}_\lambda, \lambda \in [0, \lambda_0]} $};
        
        \draw[->,blue, ultra thick] (6.1,0.3)--(6.1,3);
        \node[right, font=\footnotesize] at (6.2,1.5) {$\spfl{\mathcal{F}_{\lambda_1} + \theta I, \theta \in [0, \hat{\theta}]} $};
        
        \draw[<-,blue, ultra thick] (0,3.1)--(6,3.1);
        \node[above, font=\footnotesize] at(3,3.2) {$\spfl{\mathcal{F}_{\lambda} + \hat{\theta} I, \lambda \in [0, \lambda_1]}$};

        \node[below, color=gray] at(8,0.8) {$\spfl{\mathcal{F}_\lambda, \lambda \in [ \lambda_1,\lambda_0]} $};
        \draw[<-,blue, thick] (6,0.1)--(9,.1);

        \draw[dashed, gray] (0,3)--(6,3);
        \draw[dashed, gray] (6,0)--(6,3);
    \end{tikzpicture}
    \caption{The rectangular homotopy path in the $(\lambda, \theta)$-plane used to evaluate the spectral flow near the origin.}
    \label{fig:homotopy 2}
\end{figure}

This Lemma plays a crucial role in analyzing how the zero eigenvalue of the operator $\mathcal{L}$ responds to changes in the matrix $M$.

\begin{lem}\label{lem:sf<=0}
   Let $w_0$ be a standing pulse of \eqref{eq:r.d.eq}. Under condition (H1), there exists $\lambda_0 > 0$ such that $\mathcal{F}_{\lambda}$ is non-degenerate for all $\lambda \in (0, \lambda_0]$. Moreover, if $\int_{-\infty}^{\infty}\inpro{QMw'_0(x)}{w'_0(x)} \d x < 0$, then 
    $\spfl{\mathcal{F}_\lambda, \lambda \in [0, \lambda_0]} \leq -1$, where $z(x)=\begin{pmatrix}
    QD^{-1}w'_0(x)\\
    w'_0(x)
\end{pmatrix}$.
\end{lem}

\begin{proof}
We observe that the operator $\mathcal{L}$ has only isolated eigenvalues in the right half-plane $\mathbb{C}^+ \defeq \{z\in\mathbb{C} \mid \text{Re } z \geq 0\}$, which implies that there exists $\lambda_0 > 0$ such that $\sigma_p(\mathcal{L}) \cap (0, \lambda_0] = \emptyset$. Consequently, for all $\lambda \in (0, \lambda_0]$, the associated Hamiltonian operator $\mathcal{F}_\lambda$ for \eqref{eq:linear hamiltonian system} is non-degenerate, while $\lambda=0$ corresponds to a degenerate case.

By the translational invariance of the standing pulse $w_0$, we have $z(x) = \begin{pmatrix} QD^{-1}w'_0 \\ w'_0 \end{pmatrix} \in \ker \mathcal{F}_0$. We define the one-dimensional subspace $W \defeq \text{span}\{z(x)\}$.

To evaluate the spectral flow $\spfl{\mathcal{F}_\lambda, \lambda \in [0, \lambda_0]}$, we observe that for a sufficiently small $\hat{\theta} > 0$, the operator $\mathcal{F}_0 + \theta I$ is non-degenerate for $\theta \in (0, \hat{\theta}]$, which implies $\spfl{\mathcal{F}_0 + \theta I, \theta \in [0, \hat{\theta}]} = 0$.

According to Theorem 4.22 in \cite{robbin1995spectral}, for almost every $\hat{\theta}$, all crossings in the family $\mathcal{F}_{\lambda} + \hat{\theta}I, \lambda \in [0, \lambda_0]$ are regular and isolated. We then choose $\lambda_1 \in (0, \lambda_0]$ sufficiently small such that no crossings occur in the interval $[0, \lambda_1]$, rendering $\mathcal{F}_\lambda + \hat{\theta}I$ non-degenerate for $\lambda \in [0, \lambda_1]$.

Consider the restriction operator $(\mathcal{F}_{\lambda} + \theta I)|_W$. Its representation is given by the scalar function:
\[
f(\lambda,\theta)=\int_{-\infty}^{\infty}\inpro{(\mathcal{F}_{\lambda} + \theta I)z(x)}{z(x)}\d x.
\]

By the homotopy invariance of the spectral flow along the rectangular path shown in Figure \ref{fig:homotopy 2}, and the fact that the operator family is non-degenerate on the boundaries (except at the origin), we obtain the relation:
\begin{align}\label{eq:homotopy invariance <=0}
    &\spfl{\mathcal{F}_\lambda, \lambda \in [0, \lambda_1]} + \spfl{\mathcal{F}_{\lambda_1} + \theta I, \theta \in [0, \hat{\theta}]} \nonumber \\
    =&\spfl{\mathcal{F}_{\lambda} + \hat{\theta} I, \lambda \in [0, \lambda_1]} + \spfl{\mathcal{F}_0 + \theta I; \theta \in [0, \hat{\theta}]}.
\end{align}
Given that the right-hand side terms are zero by our choice of $\hat{\theta}$ and $\lambda_1$, we find:
\begin{align}\label{eq:sf<0-1}
    \spfl{\mathcal{F}_\lambda, \lambda \in [0, \lambda_1]} = -\spfl{\mathcal{F}_{\lambda_1} + \theta I, \theta \in [0, \hat{\theta}]}.
\end{align}
A similar logic applied to the scalar function $f$ yields:
\begin{align}\label{eq:sf<0-2}
    \spfl{f(\lambda,0), \lambda \in [0, \lambda_1]} = -\spfl{f(\lambda_1, \theta), \theta \in [0, \hat{\theta}]}.
\end{align}

Comparing the kernels of the full and restricted operators:
\begin{align}\label{eq:sf<0-3}
    \spfl{\mathcal{F}_{\lambda_1} + \theta I, \theta \in [0, \hat{\theta}]} &= \sum\limits_{\theta\in(0,\hat{\theta})}\dim\ker(\mathcal{F}_{\lambda_1} + \theta I) \nonumber \\ 
    &\geq \sum\limits_{\theta\in(0,\hat{\theta})}\dim\ker(\mathcal{F}_{\lambda_1} + \theta I)|_W \nonumber \\
    &= \spfl{f(\lambda_1, \theta), \theta \in [0, \hat{\theta}]} = -\spfl{f(\lambda, 0), \lambda \in [0, \lambda_1]}.
\end{align}

If $\int_{-\infty}^{\infty}\inpro{QMw'_0(x)}{w'_0(x)} \d x < 0$, we evaluate the crossing form of $f$ at $\lambda=0$. Since $z \in \ker \mathcal{F}_0$, the crossing form simplifies to:
\begin{align}
    \inpro{\operatorname{Cr}[f(0,0)]z}{z} &= \int_{-\infty}^{\infty} \inpro{\left.\frac{\partial}{\partial \lambda}\right|_{\lambda=0}\mathcal{F}_\lambda z}{z} \d x \nonumber \\
    &= \int_{-\infty}^{\infty} \inpro{-QMz}{z} \d x \nonumber \\
    &= -\int_{-\infty}^{\infty} \inpro{QMw'_0(x)}{w'_0(x)} \d x > 0.
\end{align}
A positive crossing form for $f(\lambda, 0)$ as $\lambda$ increases from $0$ implies:
\begin{align}\label{eq:sf<0-4}
    \spfl{f(\lambda,0), \lambda \in [0, \lambda_1]} = 1.
\end{align}
Substituting this into \eqref{eq:sf<0-3} and \eqref{eq:sf<0-1}, we conclude $\spfl{\mathcal{F}_\lambda, \lambda \in [0, \lambda_0]} \leq -1$.
\end{proof}
\begin{proof}[Proof of Theorem \ref{thm:mainly result}]
By property (2) of Proposition \ref{prop:properties of sf}, we can decompose the spectral flow as:
\[ 
\spfl{\mathcal{F}_{\lambda, 0}, \lambda \in [0, \hat{\lambda}]} = \spfl{\mathcal{F}_{\lambda, 0}, \lambda \in [0, \lambda_0]} + \spfl{\mathcal{F}_{\lambda, 0}, \lambda \in [\lambda_0, \hat{\lambda}]},
\]
so by Proposition \ref{pro:spectral flow}, we have that 
\begin{align}\label{eq:sf contain positive eigenvalue}
    \spfl{\mathcal{F}_{\lambda, 0}, \lambda \in [\lambda_0, \hat{\lambda}]}=i(w_0)-\spfl{\mathcal{F}_{\lambda, 0}, \lambda \in [0, \lambda_0]}
\end{align}

 Therefore, 
\begin{align}\label{eq:maslov<=sf at epsilon=0}
    \iCLM(\Lambda_R, E^u(\tau), \tau \in (-\infty, T]) \leq \spfl{\mathcal{F}_{\lambda, 0}, \lambda \in [\lambda_0, \hat{\lambda}]}.
\end{align}
We note that $(QD\phi'(x), \phi(x)) \in \ker(\mathcal{F}_\lambda)$ if and only if the pair $(\phi(x), \lambda)$ satisfies the weighted eigenvalue problem \eqref{eq:Weighted eigenvalue pro.}. Since $\mathcal{F}_\lambda$ is non-degenerate for all $\lambda \in (0, \lambda_0]$, invoking property (4) of Proposition \ref{prop:properties of sf} and \eqref{eq:sf contain positive eigenvalue}, we conclude that:
\begin{align}\label{eq:unstable eigenvalues lower bound}
     i(w_0)-\spfl{\mathcal{F}_{\lambda, 0}, \lambda \in [0, \lambda_0]} \leq N_+(L),
\end{align}
where $N_+(L)$ denotes the number of real, positive eigenvalues of $L$, counted according to their algebraic multiplicities.

By Lemma \ref{lem:sf<=0} and $i(w_0)\geq 0$. 
By this fact and \eqref{eq:unstable eigenvalues lower bound}, we finish  the proof.
\end{proof}

\subsection{Application to FitzHugh-Nagumo type system}\label{section 3}

  In order to investigate the stability of standing pulses of \eqref{eq:f.n.eq.1} and \eqref{eq:f.n.eq.2},
  we set
  \[
  M = \begin{pmatrix}
      1 & 0 \\
      0 & \tau
  \end{pmatrix}, \quad D=
  \begin{pmatrix}
      d & 0 \\
      0 & 1
  \end{pmatrix}, \quad \text{and} \quad 
  Q = \begin{pmatrix}
      1 & 0 \\
      0 & -1
  \end{pmatrix},
  \]
  so that $l = \min\{1, \tau\}$ and $V^-(Q) = \{0\} \oplus \R$.
  
  The equation \eqref{eq:f.n.eq.} can be written in the form of \eqref{eq:r.d.eq} by defining
  \[
  V(u, v) = \frac{1}{2} \gamma v^2 + \frac{1}{4} v^4 - u v - \frac{1}{4} u^4 + \frac{1}{3} (1+\beta) u^3 - \frac{1}{2} \beta u^2.
  \]

  By direct calculation, we have
  \begin{align}
      B(x) = \nabla^2 V(u, v) = 
      \begin{pmatrix}
          f^\prime(u) & -1 \\
          -1 & \gamma + 3 v^2
      \end{pmatrix},
  \end{align}
  so that
  \[
   B(\infty) = 
  \begin{pmatrix}
      -\beta & -1 \\
      -1 & \gamma
  \end{pmatrix}.
  \]
  We observe that $\gamma+3 v^2>0$ and $V^-(Q) = \{0\} \oplus \R$,
  Set $C_1 = \{\max\limits_{x \in \R} |f^\prime(u)|,\gamma+3\max\limits_{x \in \R} v^2)\}$, $\cic{2}=\min\{\beta,\gamma\}$ and $\cic{3}=\gamma$, a direct computation shows that 
  \begin{align}
    &\inpro{QB(x)\xi}{\xi}\leq C_1|\xi|^2,\ \text{for all } (x, \xi) \in \R \times \R^2.\\
    &\inpro{QB(\infty)\xi}{\xi}\leq -C_2|\xi|^2, \text{for all } \xi \in  \R^2.\\
    &\inpro{B(x)\xi}{\xi}\geq C_3|\xi|^2,\ \text{for all } \xi \in \{0\}\oplus \R.
  \end{align}
  hence the conditions (H1) and (H2) are readily verified.
  \begin{proof}[Proof of Theorem \ref{thm:unstable by solution}]
    Let $\phi(x) = (u'(x), v'(x))^\top$. Since the reaction-diffusion system \eqref{eq:f.n.eq.1}--\eqref{eq:f.n.eq.2} is autonomous, the pair $(\phi(x), 0)$ satisfies the weighted eigenvalue problem \eqref{eq:Weighted eigenvalue pro.} for $\lambda = 0$. Consequently, the vector-valued function $z(x) = (QD\phi'(x), \phi(x))^\top$ is a solution to the linear Hamiltonian system \eqref{eq:linear hamiltonian system}. By the definition of the unstable subspace, we have $z(\tau) \in E^u(\tau)$ for any $\tau \in \mathbb{R}$.

    At the specific point $\tau = \tau_0$, the condition $u'(\tau_0) = v''(\tau_0) = 0$ implies 
    \[
    z(\tau_0) \in \Lambda_R \cap E^u(\tau_0) \setminus \{0\},
    \]
    which demonstrates that $\Lambda_R \cap E^u(\tau_0) \neq \{0\}$. This non-trivial intersection implies that $\tau_0$ is a crossing, and therefore the stability index satisfies $i(w_0) > 0$. Applying Theorem \ref{thm:mainly result}, we conclude that the standing pulse is unstable. This completes the proof.
\end{proof}

 \begin{proof}[Proof of Theorem \ref{thm:unstablity condition for example}]
Define the critical threshold $\tau_0 \defeq \frac{\int_{-\infty}^{\infty}|u'(x)|^2\d x}{\int_{-\infty}^{\infty}|v'(x)|^2\d x}$.

For $\tau > \tau_0$,
A direct computation yields that
\begin{align}
    \int_{-\infty}^{\infty}\inpro{ QMw'_0(x)}{w'_0(x)}\d x= \int_{-\infty}^{\infty}|u'(x)|^2\d x-\tau\int_{-\infty}^{\infty}|v'(x)|^2\d x<0,
\end{align}
by Theorem \ref{thm:mainly result}, we finish the proof.
\end{proof}

\begin{proof}[Proof of Theorem \ref{thm:stablity condition for example}]
Let $w_0 = (u, v)^\top = (u, \mathcal{N}(u))^\top$ be the standing pulse solution. Recall that the linearized operator is $\mathcal{L} = M^{-\frac{1}{2}} (D \frac{\d^2}{\d x^2} - Q B(x)) M^{-\frac{1}{2}}$. A direct calculation shows that the operator $G \defeq -Q \mathcal{L}$ takes the form:
\[
G = \begin{pmatrix}
-d \frac{\d^2}{\d x^2} - f^\prime(u) & \tau^{-\frac{1}{2}} \\
\tau^{-\frac{1}{2}} & \tau^{-1} \frac{\d^2}{\d x^2} - \tau^{-1} \gamma - 3\tau^{-1} v^2
\end{pmatrix}.
\]
Identifying the components, we have $G_2 = \tau^{-1} \frac{\d^2}{\d x^2} - \tau^{-1} \gamma - 3\tau^{-1} v^2$ and the coupling term $G_3 = G_3^* = \tau^{-\frac{1}{2}}$. 

Under the condition $\tau < \gamma^2$, the assumptions of Lemma \ref{lem:all real eigenvalue} are satisfied. Combining this with the equivalence established in Lemma \ref{lem:stabilty=unstable eigenvalue}, the result follows immediately.
\end{proof}

\section{Appendix}
This appendix provides the necessary mathematical foundation for the index theory and spectral analysis used throughout this paper. We begin with a detailed examination of the spectral structure of non-self-adjoint operators, followed by a formal review of Maslov index theory and spectral flow.
\subsection{Spectral structure of the operator $\mathcal{L}$}

By analyzing the eigenvalue distribution of \( \mathcal{L} \), the stability results obtained in this section are applicable not only to FitzHugh-Nagumo type equations, but also to more general skew-gradient systems. Let \( Q^{+} \) and \( Q^{-} \) denote the orthogonal projections from \( E \) onto \( E_{+}(Q) \) and \( E_{-}(Q) \), respectively. Define 
\[
G = -Q \mathcal{L}, \quad G_1 = Q^{+} G Q^{+}, \quad G_2 = Q^{-} G Q^{-}, \quad G_3 = Q^{+} G Q^{-},
\]
that is, \( G \) is decomposed into the form:
\begin{align*}
G = \begin{pmatrix}
G_1 & G_3 \\
G_3^* & G_2
\end{pmatrix},
\end{align*}
where \( G_3^* = \bar{G}_3^T \), and \( \bar{G}_3 \) denotes the complex conjugate of \( G_3 \).

For a self-adjoint linear operator \( A \) defined on \( E \), we write \( A > 0 \) if \( \langle A \psi, \psi \rangle > 0 \) for all \( \psi \in E \setminus \{0\} \). We write \( A > \tilde{A} \) if \( A - \tilde{A} > 0 \).

\begin{lem}\label{lem:all real eigenvalue}\cite[Lemma 4.1]{CH14}
   Suppose \( -G_2 > 0 \) and \( I > G_3(-G_2)^{-2} G_3^* \), then \( \sigma(\mathcal{L}) \cap \overline{\mathbb{C}}^{+} \subset \mathbb{R} \).
\end{lem}

\begin{lem}\label{lem:stabilty=unstable eigenvalue}
    Let \( w_0 \) be a standing pulse of \eqref{eq:r.d.eq}. If \( -G_2 > 0 \) and \( I > G_3(-G_2)^{-2} G_3^* \), then \( i(w_0) = N_{+}(\mathcal{L}) \).
\end{lem}

\begin{proof}
     Recall from Proposition \ref{pro:spectral flow} and \eqref{eq:stable index=maslov index} that 
    \[
       i(w_0) = \spfl{\mathcal{F}_\lambda; \lambda \in \left[0, \hat\lambda \right]}.
    \]
    Suppose along the spectral flow there is an eigenvalue crossing at \( \mathcal{F}_\lambda \) for some \( \lambda \in \left[\lambda_0, \hat\lambda \right] \), and let \( y = (QD \phi', \phi)^\top \in \ker(\mathcal{F}_\lambda) \). It is easy to see that \( (\lambda, \phi) \) satisfies \eqref{eq:Weighted eigenvalue pro.}. Letting \( \phi_+ = Q^+ \phi \) and \( \phi_- = Q^- \phi \), we rewrite \eqref{eq:Weighted eigenvalue pro.} as
    \begin{subequations}
        \begin{align}
        	G_1 M^{1/2} \phi_+ + G_3 M^{1/2} \phi_- &= -\lambda M^{1/2} \phi_+ \label{eq:positive cross1}, \\
        	G_3^* M^{1/2} \phi_+ + G_2 M^{1/2} \phi_- &= \lambda M^{1/2} \phi_- \label{eq:positive cross2}.
        \end{align}
    \end{subequations}
    
    Solving \eqref{eq:positive cross2}, we get
    \[
    M^{1/2} \phi_- = (\lambda I - G_2)^{-1} G_3^* M^{1/2} \phi_+.
    \]
    Substituting this into \eqref{eq:positive cross1}, we obtain
    \begin{align}
    	\left\langle \frac{\d}{\d\lambda} A_\lambda(x) y, y \right\rangle_{L^2} 
    	&= \| M^{1/2} \phi_+ \|_{L^2}^2 - \| M^{1/2} \phi_- \|_{L^2}^2 \\
    	&= \langle M^{1/2} \phi_+, M^{1/2} \phi_+ \rangle_{L^2} - \langle G_3(\lambda I - G_2)^{-2} G_3^* M^{1/2} \phi_+, M^{1/2} \phi_+ \rangle_{L^2}.
    \end{align}

    Note that
    \[
    I > G_3(-G_2)^{-2} G_3^* \geq G_3(\lambda I - G_2)^{-2} G_3^*,
    \]
    for \( \lambda \geq 0 \). This implies that the sign of the crossing form is positive whenever a crossing occurs at \( \lambda \in \left[0, \hat\lambda \right] \). 

    By \eqref{eq:compute maslov by cross form}, we conclude from Lemma \ref{lem:all real eigenvalue} that
    \[
    \spfl{\mathcal{F}_\lambda; \lambda\in \left[0, \hat\lambda \right]}= \sum_{\lambda > 0} \dim E_0(\mathcal{F}_\lambda).
    \]

    This completes the proof.
\end{proof}
  \subsection{Maslov, Hörmander, Triple Index and Spectral Flow}\label{section 4}
  This final section is dedicated to recalling fundamental definitions, key results, and essential properties of the Maslov index and related invariants used throughout our analysis. Primary references include \cite{RS93,HP17,ZWZ18} and their cited works.
  
  \subsubsection{The Cappell-Lee-Miller Index}\label{sec:maslov}
  Consider the standard symplectic space $(\mathbb{R}^{2n}, \omega)$. Let $\Lag{n}$ denote the Lagrangian Grassmannian of $(\mathbb{R}^{2n}, \omega)$. For $a, b \in \mathbb{R}$ with $a < b$, define $\mathscr{P}([a, b]; \mathbb{R}^{2n})$ as the space of continuous Lagrangian pairs $L: [a, b] \to \Lag{n} \times \Lag{n}$ with compact-open topology. Following \cite{CLM94}, we recall the Maslov index for Lagrangian pairs, denoted by $\iCLM$. Intuitively, for $L = (L_1, L_2) \in \mathscr{P}([a, b]; \mathbb{R}^{2n})$, this index enumerates (with signs and multiplicities) instances $t \in [a, b]$ where $L_1(t) \cap L_2(t) \neq \{0\}$.
  
  \begin{defn}
  The $\iCLM$-index is the unique integer-valued function
  \begin{align*}
  \iCLM: \mathscr{P}([a, b]; \mathbb{R}^{2n}) \ni L \mapsto \iCLM(L(t); t \in [a, b]) \in \mathbb{Z}
  \end{align*}
  satisfying Properties I-VI in \cite[Section 1]{CLM94}.
  \end{defn}
  
  An effective approach to compute the Maslov index employs the crossing form introduced in \cite{RS93}. Let $\Lambda: [0,1] \to \Lag{n}$ be a smooth curve with $\Lambda(0) = \Lambda_0$, and $W$ a fixed Lagrangian complement of $\Lambda(t)$. For $v \in \Lambda_0$ and small $t$, define $w(t) \in W$ via $v + w(t) \in \Lambda(t)$. The quadratic form $Q(v) = \left.\frac{d}{dt}\right|_{t=0} \omega(v, w(t))$ is independent of $W$ \cite{RS93}. A crossing occurs at $t$ where $\Lambda(t)$ intersects $V \in \Lag{n}$ nontrivially. The crossing form at such $t$ is defined as
  \begin{align}\label{eq:def of crossing form in Maslov}
  \Gamma(\Lambda(t), V; t) = \left.Q\right|_{\Lambda(t) \cap V}.
  \end{align}
  
  A crossing is regular if its form is nondegenerate. For quadratic form $Q$, let $\sign(Q) = \coiMor(Q) - \iMor(Q)$ denote its signature. From \cite{ZL99}, if $\Lambda(t)$ has only regular crossings with $V$, then
  \begin{align}\label{eq:compute maslov by cross form}
  \begin{split}
	  \iCLM(V, \Lambda(t); t \in [a,b]) = &\coiMor(\Gamma(\Lambda(a), V; a)) \\
  &+ \sum_{a<t<b} \sign \Gamma(\Lambda(t), V; t) - \iMor(\Gamma(\Lambda(b), V; b)).
  \end{split}
  \end{align}
  For the sake of the reader, we list a couple of properties of the $\iCLM$-index that we shall use throughout the paper.
  \begin{itemize}
	  \item \textbf{(Reversal)} Let $L:=\left(L_1, L_2\right) \in \mathscr{P}\left([a, b] ; \mathbb{R}^{2 n}\right)$. Denoting by $\widehat{L} \in$ $\mathscr{P}\left([-b,-a] ; \mathbb{R}^{2 n}\right)$ the path traveled in the opposite direction, and by setting $\widehat{L}:=\left(L_1(-s), L_2(-s)\right)$, we obtain
  \begin{align*}
	  \iCLM(\widehat{L} ;[-b,-a])=-\iCLM(L ;[a, b])
	  \end{align*}
	  \item \textbf{(Stratum homotopy relative to the ends)} Given a continuous map
	  $L:[a, b] \ni s \rightarrow L(s) \in \mathscr{P}\left([a, b] ; \mathbb{R}^{2 n}\right)$ where $L(s)(t):=\left(L_1(s, t), L_2(s, t)\right)$
	  such that $\dim L_1(s, a) \cap L_2(s, a)$ and $\dim L_1(s, b) \cap L_2(s, b)$ are both constant, and then,	
  \begin{align*}
	  \iCLM(L(0) ;[a, b])=\iCLM(L(1) ;[a, b])
	  \end{align*}
  \end{itemize}

  \subsubsection{Triple Index and Hörmander Index}
  We summarize key concepts about the triple and Hörmander indices, following \cite{ZWZ18}. For isotropic subspaces $\alpha, \beta, \delta$ in $(\mathbb{R}^{2n}, \omega)$, define the quadratic form
  \begin{align}\label{eq:definition q}
  \mathscr{Q} \coloneqq \mathscr{Q}(\alpha, \beta; \delta): \alpha \cap (\beta + \delta) \to \mathbb{R},\quad \mathscr{Q}(x_1, x_2) = \omega(y_1, z_2)
  \end{align}
  where $x_j = y_j + z_j \in \alpha \cap (\beta + \delta)$ with $y_j \in \beta$, $z_j \in \delta$. For Lagrangian subspaces $\alpha, \beta, \delta$, \cite[Lemma 3.3]{ZWZ18} gives
  \begin{align*}
  \ker \mathscr{Q}(\alpha, \beta; \delta) = \alpha \cap \beta + \alpha \cap \delta.
  \end{align*}
  
  \begin{defn}
  For Lagrangians $\alpha, \beta, \kappa$ in $(\mathbb{R}^{2n}, \omega)$, the triple index is
  \begin{align}\label{eq:def of trip 1}
  \iota(\alpha, \beta, \kappa) = \iMor(\mathscr{Q}(\alpha, \delta; \beta)) + \iMor(\mathscr{Q}(\beta, \delta; \kappa)) - \iMor(\mathscr{Q}(\alpha, \delta; \kappa))
  \end{align}
  where $\delta$ satisfies $\delta \cap \alpha = \delta \cap \beta = \delta \cap \kappa = \{0\}$.
  \end{defn}
  
  By \cite[Lemma 3.13]{ZWZ18}, this index also satisfies
  \begin{align}\label{eq:def of trip 2}
  \iota(\alpha, \beta, \kappa) = \iMor(\mathscr{Q}(\alpha, \beta; \kappa)) + \dim(\alpha \cap \kappa) - \dim(\alpha \cap \beta \cap \kappa).
  \end{align}
  
  The Hörmander index measures the difference between Maslov indices relative to different Lagrangians. For paths $\Lambda, V \in \mathscr{C}^0([0,1], \Lag{n})$ with endpoints $\Lambda(0)=\Lambda_0$, $\Lambda(1)=\Lambda_1$, $V(0)=V_0$, $V(1)=V_1$:
  
  \begin{defn}\label{def:hormand}
  The Hörmander index is
  \begin{align}\label{eq:def of Hormand}
  s(\Lambda_0, \Lambda_1; V_0, V_1) &= \iCLM(V_1, \Lambda(t); t \in [0,1]) - \iCLM(V_0, \Lambda(t); t \in [0,1]) \\
  &= \iCLM(V(t), \Lambda_1; t \in [0,1]) - \iCLM(V(t), \Lambda_0; t \in [0,1]).
  \end{align}
  \end{defn}
  
  \begin{rem}
  Homotopy invariance ensures Definition \ref{def:hormand} is well-posed (cf. \cite{RS93}).
  \end{rem}
  
  For four Lagrangians $\lambda_1, \lambda_2, \kappa_1, \kappa_2$, \cite[Theorem 1.1]{ZWZ18} establishes:
  \begin{align}\label{eq:compute Hormand by trip}
  s(\lambda_1, \lambda_2; \kappa_1, \kappa_2) = \iota(\lambda_1, \lambda_2, \kappa_2) - \iota(\lambda_1, \lambda_2, \kappa_1) = \iota(\lambda_1, \kappa_1, \kappa_2) - \iota(\lambda_2, \kappa_1, \kappa_2).
  \end{align}
  \begin{lem}\cite[Lemma A.6]{HPWX20}\label{lem:compute maslov by triple index}
	Let $\Lambda_1$ and $\Lambda_2$ be two continuous paths in $\Lag{n}$ with $t \in[0,1]$ and we assume that $\Lambda_1(t)$ and $\Lambda_2(t)$ are both transversal to the (fixed) Lagrangian subspace $\Lambda$. Then we get
  \begin{align*}
  \iCLM\left(\Lambda_1(t), \Lambda_2(t) ; t \in[0,1]\right)=\iota\left(\Lambda_2(1), \Lambda_1(1) ; \Lambda\right)-\iota\left(\Lambda_2(0), \Lambda_1(0) ; \Lambda\right)
  \end{align*}
  \end{lem}

  \subsubsection{Spectral Flow}
  Introduced by Atiyah-Patodi-Singer \cite{APS76}, spectral flow measures eigenvalue crossings. Let $E$ be a real separable Hilbert space, and $\mathscr{CF}^{sa}(E)$ denote closed self-adjoint Fredholm operators with gap topology. For continuous $A: [0,1] \to \mathscr{CF}^{sa}(E)$, the spectral flow $\spfl{A_t; t \in [0,1]}$ counts signed eigenvalue crossings through $-\epsilon$ ($\epsilon > 0$ small).
  
  For each $A_t$, consider the orthogonal decomposition
  \begin{align*}
  E = E_{-}(A_t) \oplus E_0(A_t) \oplus E_{+}(A_t).
  \end{align*}
  Let $P_t$ be the orthogonal projector onto $E_0(A_t)$. At crossing $t_0$ where $E_0(A_{t_0}) \neq \{0\}$, define the crossing form
  \begin{align}\label{eq:def of cross form in sf}
  \operatorname{Cr}[A_{t_0}] \coloneqq P_{t_0} \frac{\partial}{\partial t} P_{t_0}: E_0(A_{t_0}) \to E_0(A_{t_0}).
  \end{align}
  
  A crossing is regular if $\operatorname{Cr}[A_{t_0}]$ is nondegenerate. Define
  \begin{align*}
  \sgn(\operatorname{Cr}[A_{t_0}]) \coloneqq \dim E_{+}(\operatorname{Cr}[A_{t_0}]) - \dim E_{-}(\operatorname{Cr}[A_{t_0}]).
  \end{align*}
  
  Assuming regular crossings, the spectral flow becomes
  \begin{align}\label{eq:compute sf by crossing form}
  \spfl{A_t; t \in [0,1]} = \sum_{t_0 \in \mathcal{S}_*} \sgn(\operatorname{Cr}[A_{t_0}]) - \dim E_{-}(\operatorname{Cr}[A_0]) + \dim E_{+}(\operatorname{Cr}[A_1])
  \end{align}
  where $\mathcal{S}_* = \mathcal{S} \cap (a,b)$ contains crossings in $(a,b)$.
  
  For the sake of the reader we list some properties of the spectral flow that we shall frequently use in the paper.
  \begin{prop}\label{prop:properties of sf}
    \begin{itemize}
	  \item[(1)]  Given a continuous map
	  \begin{align}
	  \bar{A}:[0,1] \rightarrow \mathscr{C}^0\left([a, b] ; \mathscr{C} \mathscr{F}^{s a}(E)\right) \text { where } \bar{A}(s)(t):=\bar{A}^s(t)
	  \end{align}	
	  such that $\operatorname{dim} \operatorname{ker} \bar{A}^s(a)$ and $\operatorname{dim} \operatorname{ker} \bar{A}^s(b)$ are both independent on $s$, then	
	  \begin{align}\label{eq:Stratum homotopy relative to the ends}
	  \operatorname{sf}\left(\bar{A}_t^0 ; t \in[a, b]\right)=\operatorname{sf}\left(\bar{A}_t^1 ; t \in[a, b]\right)
	  \end{align}
	  \item[(2)]  If $A^1, A^2 \in \mathscr{C}^0\left([a, b] ; \mathscr{C} \mathscr{F}^{s a}(E)\right)$ are such that $A^1(b)=A^2(a)$, then
	  \begin{align}\label{eq:Path additivity}
	  \operatorname{sf}\left(A_t^1 * A_t^2 ; t \in[a, b]\right)=\operatorname{sf}\left(A_t^1 ; t \in[a, b]\right)+\operatorname{sf}\left(A_t^2 ; t \in[a, b]\right)
	  \end{align}	
	  where $*$ denotes the usual catenation between the two paths.
	  \item[(3)] If $A \in \mathscr{C}^0([a, b] ; \operatorname{GL}(E))$, then 
      \begin{align}\label{eq:Nullity}
        \operatorname{sf}\left(A_t ; t \in[a, b]\right)=0.
      \end{align}
      \item[(4)] If $\hat{\Omega}=\left\{t \mid 0 \leq t \leq 1\right.$ and $\left.\dim E_0\left(A_t\right) \neq 0\right\}$ then
      \begin{align}\label{eq:bounded for sf}
      \left|\spfl{A_t; t\in[0,1]}\right| \leq \sum_{t \in \hat{\Omega}} \dim E_0\left(A_t\right) .
      \end{align}
      
  \end{itemize}

  \end{prop}

  \subsubsection{Spectral flow formula for general Hamiltonian system}
  Borrowing the notation of \cite{HP17} for $\lambda \in[0,1]$ we denote by $\gamma_{(\tau, \lambda)}$ be the (primary) fundamental solution of the following linear Hamiltonian system
  \begin{align}\label{eq:linear h.s.eq.}
  \begin{cases}
	  \dot{\gamma}(t)=J B_\lambda(t) \gamma(t), \quad t \in \mathbb{R} \\
  \gamma(\tau)=I
  \end{cases}
  \end{align}
  We introduce the following condition
  
  We define, respectively, the stable and unstable subspaces as follows
  \begin{align*}
  E_\lambda^s(\tau):=\left\{v \in \mathbb{R}^{2 n} \mid \lim _{t \rightarrow+\infty} \gamma_{(\tau, \lambda)}(t) v=0\right\} \quad \text { and } \quad E_\lambda^u(\tau):=\left\{\left.v \in \mathbb{R}^{2 n}\right|_{t \rightarrow-\infty} \gamma_{(\tau, \lambda)}(t) v=0\right\}
  \end{align*}
  We observe that, for every $(\lambda, \tau) \in[0,1] \times \mathbb{R}, E_\lambda^s(\tau), E_\lambda^u(\tau) \in L(n)$. (For further details, we refer the interested reader to \cite{CH07,HP17}and references therein). Setting
  \begin{align*}
  \begin{array}{l} 
  E_\lambda^s(+\infty):=\left\{v \in \mathbb{R}^{2 n} \mid \lim _{t \rightarrow+\infty} \exp \left(t B_\lambda(\infty)\right) v=0\right\} \\
  E_\lambda^u(-\infty):=\left\{v \in \mathbb{R}^{2 n} \mid \lim _{t \rightarrow-\infty} \exp \left(t B_\lambda(\infty)\right) v=0\right\}
  \end{array}
  \end{align*}
  and assuming that condition (L1) holds, then we get that
  \begin{align*}
  \lim _{\tau \rightarrow+\infty} E_\lambda^s(\tau)=E_\lambda^s(+\infty) \quad \text { and } \quad \lim _{\tau \rightarrow-\infty} E_\lambda^u(\tau)=E_\lambda^u(-\infty)
  \end{align*}
  where the convergence is meant in the gap (norm) topology of the Lagrangian Grassmannian. (Cfr. \cite{AO01} for further details).

\begin{prop}\label{pro:Hp sf=malsov inde }
    Under the previous notation and if condition (L1) holds, then 
    \begin{align}
        \iCLM(E^s_\lambda(0),E_\lambda^u(0),\lambda\in[0,1])=-\spfl{\mathcal{F}_\lambda,\lambda\in[0,1]}.
    \end{align}
\end{prop}

  \begin{prop}\label{prop:HP sf formula}
	  Under the previous notation and if condition (L1) holds, then \\
      (i) we have that 
	  \begin{align}\label{eq:HP sf formula}
				  -\spfl{\mathcal{F}_\lambda;\lambda\in[0,1]}=\iCLM(E_0^s(\tau),E_0^u(-\tau);[0,+\infty))-\iCLM(E^s_1(\tau),E_1^u(-\tau);[0,+\infty)),
		 \end{align}
         where $\mathcal{F}_\lambda:=-J\frac{\d}{\d x}-B_\lambda(x).$

(ii)For any $T>0$ and $\Lambda\in\Lag{2n}$, we have that
\begin{align}
    &-\spfl{\mathcal{F}_{T,\lambda};\lambda\in[0,1]}\\=&\iCLM(E_0^s(\Lambda,E_0^u(\tau);\tau\in(-\infty,T])-\iCLM(\Lambda,E_1^u(\tau);\tau\in(-\infty,T])
    -\iCLM(\Lambda_0,E^u_\lambda(-\infty)),
\end{align}
where $\mathcal{F}_{T,\lambda}=-J\frac{\d}{\d x}-B_\lambda(x)$ with $\dom \mathcal{F}_{T,\lambda}=\left\{   z(t)\in W^{1,2}((-\infty,T]))|z(T)\in \Lambda \right\}$
	  \end{prop}
  \begin{proof}
	  The proof directly follows by \cite[Theorem 1]{HP17} and Reversal property of Maslov index.
  \end{proof}

%


\begin{thebibliography}{34}
    \providecommand{\natexlab}[1]{#1}
    \providecommand{\url}[1]{\texttt{#1}}
    \expandafter\ifx\csname urlstyle\endcsname\relax
      \providecommand{\doi}[1]{doi: #1}\else
      \providecommand{\doi}{doi: \begingroup \urlstyle{rm}\Url}\fi
    
    \bibitem[Abbondandolo and Majer(2003)]{AM03}
    Alberto Abbondandolo and Pietro Majer.
    \newblock Ordinary differential operators in hilbert spaces and fredholm pairs.
    \newblock \emph{Mathematische Zeitschrift}, 243\penalty0 (3):\penalty0 525--562, 2003.
    
    \bibitem[Agarwal and O'Regan(2012)]{AO01}
    Ravi~P Agarwal and Donal O'Regan.
    \newblock \emph{Infinite interval problems for differential, difference and integral equations}.
    \newblock Springer Science \& Business Media, 2012.
    
    \bibitem[Arnol'd(1985)]{arnol1985sturm}
    Vladimir~Igorevich Arnol'd.
    \newblock The sturm theorems and symplectic geometry.
    \newblock \emph{Functional analysis and its applications}, 19\penalty0 (4):\penalty0 251--259, 1985.
    
    \bibitem[Atiyah et~al.(1976)Atiyah, Patodi, and Singer]{APS76}
    Michael~F Atiyah, Vijay~K Patodi, and Isadore~M Singer.
    \newblock Spectral asymmetry and riemannian geometry. iii.
    \newblock In \emph{Mathematical Proceedings of the Cambridge Philosophical Society}, volume~79, pages 71--99. Cambridge University Press, 1976.
    
    \bibitem[Barker et~al.(2018)Barker, Humpherys, Lyng, and Lytle]{barker2018evans}
    Blake Barker, Jeffrey Humpherys, Gregory Lyng, and Joshua Lytle.
    \newblock Evans function computation for the stability of travelling waves.
    \newblock \emph{Philosophical Transactions of the Royal Society A: Mathematical, Physical and Engineering Sciences}, 376\penalty0 (2117):\penalty0 20170184, 2018.
    
    \bibitem[Beck et~al.(2018)Beck, Cox, Jones, Latushkin, McQuighan, and Sukhtayev]{BCJLM18}
    Margaret Beck, Graham Cox, CKRT Jones, Yuri Latushkin, Kelly McQuighan, and Alim Sukhtayev.
    \newblock Instability of pulses in gradient reaction--diffusion systems: a symplectic approach.
    \newblock \emph{Philosophical Transactions of the Royal Society A: Mathematical, Physical and Engineering Sciences}, 376\penalty0 (2117):\penalty0 20170187, 2018.
    
    \bibitem[Bose and Jones(1995)]{bose1995stability}
    Amitabha Bose and Christopher~KRT Jones.
    \newblock Stability of the in-phase travelling wave solution in a pair of coupled nerve fibers.
    \newblock \emph{Indiana University Mathematics Journal}, pages 189--220, 1995.
    
    \bibitem[Bott(1956)]{bott1956iteration}
    Raoul Bott.
    \newblock On the iteration of closed geodesics and the sturm intersection theory.
    \newblock \emph{Communications on Pure and Applied Mathematics}, 9\penalty0 (2):\penalty0 171--206, 1956.
    
    \bibitem[Cappell et~al.(1994)Cappell, Lee, and Miller]{CLM94}
    Sylvain~E Cappell, Ronnie Lee, and Edward~Y Miller.
    \newblock On the maslov index.
    \newblock \emph{Communications on Pure and Applied Mathematics}, 47\penalty0 (2):\penalty0 121--186, 1994.
    
    \bibitem[Chen and Hu(2007)]{CH07}
    Chao-Nien Chen and Xijun Hu.
    \newblock Maslov index for homoclinic orbits of hamiltonian systems.
    \newblock \emph{Annales de l'Institut Henri Poincar{\'e} C}, 24\penalty0 (4):\penalty0 589--603, 2007.
    
    \bibitem[Chen and Hu(2014)]{CH14}
    Chao-Nien Chen and Xijun Hu.
    \newblock Stability analysis for standing pulse solutions to fitzhugh--nagumo equations.
    \newblock \emph{Calculus of Variations and Partial Differential Equations}, 49:\penalty0 827--845, 2014.
    
    \bibitem[Choi and Lee(2021)]{choi2021existence}
    Yung-Sze Choi and Jieun Lee.
    \newblock Existence of standing pulse solutions to a skew-gradient system.
    \newblock \emph{Journal of Differential Equations}, 302:\penalty0 185--221, 2021.
    
    \bibitem[Conley(2006)]{conley2006oscillation}
    C~Conley.
    \newblock An oscillation theorem for linear systems with more than one degree of freedom.
    \newblock In \emph{Conference on the Theory of Ordinary and Partial Differential Equations: Held in Dundee/Scotland, March 28--31, 1972}, pages 232--235. Springer, 2006.
    
    \bibitem[Cornwell and Jones(2018)]{cornwell2018existence}
    Paul Cornwell and Christopher~KRT Jones.
    \newblock On the existence and stability of fast traveling waves in a doubly diffusive fitzhugh--nagumo system.
    \newblock \emph{SIAM Journal on Applied Dynamical Systems}, 17\penalty0 (1):\penalty0 754--787, 2018.
    
    \bibitem[Evans(1972)]{evans1972nerve}
    John~W Evans.
    \newblock Nerve axon equations: Iii stability of the nerve impulse.
    \newblock \emph{Indiana University Mathematics Journal}, 22\penalty0 (6):\penalty0 577--593, 1972.
    
    \bibitem[Flores(1991)]{flores1991stability}
    Gilberto Flores.
    \newblock Stability analysis for the slow travelling pulse of the fitzhugh--nagumo system.
    \newblock \emph{SIAM journal on mathematical analysis}, 22\penalty0 (2):\penalty0 392--399, 1991.
    
    \bibitem[Hu and Portaluri(2017)]{HP17}
    Xijun Hu and Alessandro Portaluri.
    \newblock Index theory for heteroclinic orbits of hamiltonian systems.
    \newblock \emph{Calculus of Variations and Partial Differential Equations}, 56\penalty0 (6):\penalty0 167, 2017.
    
    \bibitem[Hu et~al.(2020)Hu, Portaluri, Wu, and Xing]{HPWX20}
    Xijun Hu, Alessandro Portaluri, Li~Wu, and Qin Xing.
    \newblock Morse index theorem for heteroclinic orbits of lagrangian systems.
    \newblock \emph{arXiv preprint arXiv:2004.08643}, 2020.
    
    \bibitem[Jones et~al.(1990)Jones, Gardner, and Alexander]{jones1990topological}
    C~Jones, Robert Gardner, and J~Alexander.
    \newblock A topological invariant arising in the stability analysis of travelling waves.
    \newblock 1990.
    
    \bibitem[Jones(1985)]{jones1985instability}
    Christopher~KRT Jones.
    \newblock \emph{Instability of standing waves for nonlinear Schr{\"o}dinger type equations}.
    \newblock Brown University. Lefschetz Center for Dynamical Systems. Division of~…, 1985.
    
    \bibitem[Kapitula et~al.(2013)Kapitula, Promislow, et~al.]{kapitula2013spectral}
    Todd Kapitula, Keith Promislow, et~al.
    \newblock \emph{Spectral and dynamical stability of nonlinear waves}, volume 457.
    \newblock Springer, 2013.
    
    \bibitem[Pego and Weinstein(1992)]{pego1992eigenvalues}
    Robert~L Pego and Michael~I Weinstein.
    \newblock Eigenvalues, and instabilities of solitary waves.
    \newblock \emph{Philosophical Transactions of the Royal Society of London. Series A: Physical and Engineering Sciences}, 340\penalty0 (1656):\penalty0 47--94, 1992.
    
    \bibitem[Robbin and Salamon(1993)]{RS93}
    Joel Robbin and Dietmar Salamon.
    \newblock The maslov index for paths.
    \newblock \emph{Topology}, 32\penalty0 (4):\penalty0 827--844, 1993.
    \bibitem{robbin1995spectral}
    Joel Robbin and Dietmar Salamon.
    \newblock The spectral flow and the Maslov index
    \newblock \emph{Bulletin of the London Mathematical Society}, 27\penalty0(124): \penalty0 1-33, 1995.
    
    \bibitem[Robbins(1992)]{robbins1992winding}
    Jonathan Robbins.
    \newblock Winding number formula for maslov indices.
    \newblock \emph{Chaos: An Interdisciplinary Journal of Nonlinear Science}, 2\penalty0 (1):\penalty0 145--147, 1992.
    
    \bibitem[Sandstede(2002)]{sandstede2002stability}
    Bj{\"o}rn Sandstede.
    \newblock Stability of travelling waves.
    \newblock In \emph{Handbook of dynamical systems}, volume~2, pages 983--1055. Elsevier, 2002.
    
    \bibitem[Turing(1990)]{turing1990chemical}
    Alan~Mathison Turing.
    \newblock The chemical basis of morphogenesis.
    \newblock \emph{Bulletin of mathematical biology}, 52:\penalty0 153--197, 1990.
    
    \bibitem[van Heijster et~al.(2008)van Heijster, Doelman, and Kaper]{van2008pulse}
    Peter van Heijster, Arjen Doelman, and Tasso~J Kaper.
    \newblock Pulse dynamics in a three-component system: stability and bifurcations.
    \newblock \emph{Physica D: Nonlinear Phenomena}, 237\penalty0 (24):\penalty0 3335--3368, 2008.
    
    \bibitem[Xing(2021)]{X21}
    Qin Xing.
    \newblock Index theory for traveling waves in reaction diffusion systems with skew gradient structure.
    \newblock \emph{Proceedings of the American Mathematical Society}, 149\penalty0 (7):\penalty0 2891--2909, 2021.
    
    \bibitem[Yanagida(1985)]{yanagida1985stability}
    Eiji Yanagida.
    \newblock Stability of fast travelling pulse solutions of the fitzhugh—nagumo equations.
    \newblock \emph{Journal of Mathematical Biology}, 22:\penalty0 81--104, 1985.
    
    \bibitem[Yanagida(2002)]{yanagida2002mini}
    Eiji Yanagida.
    \newblock Mini-maximizers for reaction-diffusion systems with skew-gradient structure.
    \newblock \emph{Journal of Differential Equations}, 179\penalty0 (1):\penalty0 311--335, 2002.
    
    \bibitem[Yang and Xing(2025)]{yang2025morse}
    Ran Yang and Qin Xing.
    \newblock Morse index theorem for sturm-liouville operators on the real line.
    \newblock \emph{arXiv preprint arXiv:2504.05091}, 2025.
    
    \bibitem[Zhou et~al.(2018)Zhou, Wu, and Zhu]{ZWZ18}
    Yuting Zhou, Li~Wu, and Chaofeng Zhu.
    \newblock H{\"o}rmander index in finite-dimensional case.
    \newblock \emph{Frontiers of Mathematics in China}, 13:\penalty0 725--761, 2018.
    
    \bibitem[Zhu and Long(1999)]{ZL99}
    Chaofeng Zhu and Yiming Long.
    \newblock Maslov-type index theory for symplectic paths and spectral flow (i).
    \newblock \emph{Chinese Annals of Mathematics}, 20\penalty0 (04):\penalty0 413--424, 1999.
    
    \bibitem[Zumbrun et~al.(2005)Zumbrun, Jenssen, and Lyng]{zumbrun2005stability}
    Kevin Zumbrun, Helge~Kristian Jenssen, and Gregory Lyng.
    \newblock Stability of large-amplitude shock waves of compressible navier--stokes equations.
    \newblock In \emph{Handbook of mathematical fluid dynamics}, volume~3, pages 311--533. Elsevier, 2005.
    
    \end{thebibliography}

\vspace{1cm}
\noindent
\textsc{Jing Li}\\
School of Mathematics and Statistics, Linyi University\\
 Linyi, Shandong \\
The People's Republic of China\\
E-mail:\ \email{lijing241301@126.com}

\vspace{1cm}
\noindent
\textsc{Dr. Qin Xing}\\
School of Mathematics and Statistics, Linyi University\\
 Linyi, Shandong \\
The People's Republic of China\\
E-mail:\ \email{xingqin@lyu.edu.cn}

\vspace{1cm}
\noindent
\textsc{Dr. Ran Yang}\\
School of Science\\
East China  University of Technology\\
Nanchang, Jiangxi, 330013\\
The People's Republic of China \\
E-mail:\ \email{201960124@ecut.edu.cn}

\end{document}